\documentclass[11pt]{article}

\usepackage{amsmath,amssymb}
\usepackage{amsthm}
\usepackage{mathrsfs}
\usepackage{mathtools}
\usepackage{enumitem}

\setlist{topsep=3pt,itemsep=-0.5ex}

\usepackage[a4paper,margin=28mm]{geometry}

\usepackage[%
bookmarksopen=true,
bookmarksopenlevel=3,
hidelinks,
bookmarksnumbered 
]{hyperref}

\theoremstyle{definition}
\newtheorem{dfn}{Definition}[section]
\newtheorem{thm}[dfn]{Theorem}
\newtheorem{lem}[dfn]{Lemma}
\newtheorem{prop}[dfn]{Proposition}
\newtheorem{cor}[dfn]{Corollary}
\newtheorem{ex}[dfn]{Example}
\newtheorem{rem}[dfn]{Remark}
\newtheorem{assum}[dfn]{Assumption}

\newcommand{\R}{\mathbb{R}}
\newcommand{\E}{\mathbb{E}}
\newcommand{\Prob}{\mathbb{P}}
\newcommand{\supp}{\mathop{\mathrm{supp}}}
\newcommand{\tr}{\mathop{\mathrm{tr}}}
\newcommand{\divg}{\mathop{\mathrm{div}}}
\DeclareMathOperator*{\esssup}{ess\,sup}

\makeatletter
\newcommand{\roverline}[1]{\mathpalette\@roverline{#1}}
\newcommand{\@roverline}[2]{\overline{#1#2}}
\makeatother

\DeclareFontFamily{U}{MYmathx}{}
\DeclareFontShape{U}{MYmathx}{m}{n}{<-> mathx10}{}
\DeclareSymbolFont{MYmathx}{U}{MYmathx}{m}{n}
\DeclareMathAccent{\widecheck}{0}{MYmathx}{"71}

\makeatletter

\@addtoreset{equation}{section}
\makeatother

\makeatletter
\renewenvironment{proof}[1][\proofname]{\par
	\pushQED{\qed}%
	\normalfont \topsep6\p@\@plus6\p@\relax
	\trivlist\item[\hskip\labelsep\textbf{#1\@addpunct{.}}] 
	\ignorespaces
}{%
	\popQED\endtrivlist\@endpefalse
}

\title{A norm equivalence result for stochastic differential equations with locally Lipschitz coefficients}
\date{}
\author{Kyo Yamazaki\thanks{Graduate School of Engineering Science, The University of Osaka, Japan.}}

\begin{document}
	\maketitle
	\begin{abstract}
		We establish two-sided weighted integrability estimates, often referred to as a norm equivalence result, for stochastic differential equations (SDEs) with locally Lipschitz coefficients.
		As a key ingredient in our approach, we also derive an SDE satisfied by the inverse stochastic flow under reduced regularity assumptions in the globally Lipschitz setting.
	\end{abstract}
	
	
	
	\section*{Introduction}
	This paper is concerned with the following stochastic differential equation (SDE)
	\begin{align*}
		X_s^{t,x}=x+\int_t^s b(r,X_r^{t,x})\,dr+\int_t^s \sigma(r,X_r^{t,x})\,dW_r,
		\quad s\in[t,T],
	\end{align*}
	where $b:[0,T]\times\R^d\to\R^d$ and $\sigma:[0,T]\times\R^d\to\R^{d\times d'}$, and $(W_t)_{t\in[0,T]}$ is a $d'$-dimensional Brownian motion.
	Our main objective is to establish two-sided weighted integrability estimates of the form
	\begin{align}
		&c\int_{\R^d}|\varphi(x)|\rho(x)\,dx
		\le \int_{\R^d}\E\big[|\varphi(X^{t,x}_s)|\big]\rho(x)\,dx
		\le C\int_{\R^d}|\varphi(x)|\rho(x)\,dx,
		\label{eq:intro.norm_eq_result1}\\
		&\begin{aligned}
			c\int_{\R^d}\int_t^T|\psi(s,x)|\rho(x)\,ds\,dx
			&\le \int_{\R^d}\int_t^T\E\big[|\psi(s,X^{t,x}_s)|\big]\rho(x)\,ds\,dx \\
			&\qquad\qquad\le C\int_{\R^d}\int_t^T|\psi(s,x)|\rho(x)\,ds\,dx, 
		\end{aligned}\label{eq:intro.norm_eq_result2}
	\end{align}
	for arbitrary Borel functions $\varphi:\R^d\to\R$ and $\psi:[0,T]\times\R^d\to\R$.
	Here $\rho:\R^d\to(0,\infty)$ is a suitable weight function, and the constants $c,C>0$ are independent of $\varphi$ and $\psi$.
	Results of this type are often referred to as a \emph{norm equivalence result} or a \emph{norm equivalence principle} (e.g., \cite{BarlesLesigne1997,BallyMatoussi2001,OuknineTurpin2006,GobetLabart2010,FengWangZhao2018}; see also \cite{DelbaenQiuTang2015}).
	
	\medskip
	\noindent
	\textbf{Motivation.}
	Norm equivalence plays a fundamental role in the probabilistic approach to partial differential equations (PDEs), especially in the study of Sobolev weak solutions and their stochastic representations via BSDEs/FBSDEs.
	For semilinear and quasilinear parabolic equations, the nonlinear Feynman--Kac formula connects the PDE solution $u$ (and typically $\sigma^\top\nabla u$) with suitable components of the associated BSDE/FBSDE (see, e.g., \cite{PardouxPeng1992}).
	Analogous probabilistic representations have been developed for Sobolev weak solutions; see, for example, \cite{BarlesLesigne1997,BallyMatoussi2001,OuknineTurpin2006,ZhangF2009,FengWangZhao2018}.
	In this framework, norm equivalence provides an essential bridge between function-space integrability of $(u,\sigma^\top\nabla u)$ and stochastic integrability of the processes obtained by composing these objects with the forward diffusion.
	More concretely, combining a (nonlinear) Feynman--Kac representation with suitable a priori estimates for BSDEs/FBSDEs, one can derive weighted $L^p$-type bounds for the corresponding PDE solutions;
	conversely, analytic integrability estimates for weak PDE solutions can be used to control the associated BSDE/FBSDE solutions in appropriate stochastic spaces (see, for instance, \cite{BarlesLesigne1997,BallyMatoussi2001,OuknineTurpin2006}).
	In addition, when $\varphi$ is defined only up to Lebesgue-null sets, \eqref{eq:intro.norm_eq_result1} ensures that the composition $\varphi(X_s^{t,x})$ is still well defined in the $\Prob\otimes dx$-a.e.\ sense and can be handled rigorously.
	Such compositions naturally arise, for example, when one identifies $u(s,X_s^{t,x})$ and $(\sigma^\top\nabla u)(s,X_s^{t,x})$ in the study of Sobolev solutions and their probabilistic interpretations (e.g., \cite{BallyMatoussi2001,OuknineTurpin2006,FengWangZhao2018}).
	
	\medskip
	\noindent
	\textbf{Related literature.}
	The norm equivalence principle was first established by Barles and Lesigne~\cite{BarlesLesigne1997} for $\rho=1$ in the context of Sobolev solutions of semilinear parabolic PDEs,
	where analytic properties of the associated PDE are combined with a probabilistic representation of its solution.
	Bally and Matoussi~\cite{BallyMatoussi2001} later proved weighted versions for exponential-type weights $\rho=e^{F}$ under smoothness assumptions on the coefficients.
	Their argument is primarily probabilistic and is centered on estimates of the Jacobian determinant of the inverse flow, which naturally arise in change-of-variable arguments.
		
	Ouknine and Turpin~\cite{OuknineTurpin2006} proposed a related approach aimed at reducing the required smoothness by incorporating ideas connected with Dirichlet forms in the spirit of
	Bouleau--Hirsch~\cite{BouleauHirsch1988}.
	Their proof follows the same broad strategy of estimating Jacobians of inverse flows.
	
	A different route to norm equivalence relies on transition density estimates.
	For instance, Gobet and Labart~\cite{GobetLabart2010} derived a norm equivalence estimate for exponential spatial weights of the form $\rho(x)=e^{-\mu|x|}$ (up to a mild time-dependent factor),
	using two-sided Gaussian bounds for transition densities under a uniform ellipticity condition.
	Compared with the inverse-flow approach, this method can substantially relax the spatial regularity requirements on the coefficients and avoids explicit Jacobian estimates,
	but it may be viewed as closer in spirit to PDE techniques through the use of Aronson-type bounds.
	
	From the viewpoint of time-dependent coefficients,
	Delbaen, Qiu and Tang~\cite{DelbaenQiuTang2015} obtained a norm equivalence estimate in a setting where the Lipschitz constant of the drift may depend on time,
	motivated by applications to FBSDEs with Sobolev-regular coefficients.
	Such time-inhomogeneous controls appear to be relatively rare in the existing literature on norm equivalence.
	
	Several extensions and variants of the norm equivalence principle have also been developed.
	Zhang and Zhao~\cite{ZhangZhao2007} introduced a \emph{generalized equivalence of norm principle}
	for certain classes of random functions and employed it in the study of finite and infinite horizon BDSDEs and their links to weak solutions of SPDEs.
	On the other hand, Feng, Wang, and Zhao~\cite{FengWangZhao2018} formulated norm equivalence-type arguments in a coupled FBSDE framework associated with quasilinear equations.
	Furthermore, Matoussi, Sabbagh and Zhou~\cite{MatoussiSabbaghZhou2015} extended the inverse-flow strategy to jump settings and established a corresponding norm equivalence result.	
	
	\medskip
	\noindent
	\textbf{A remark on the inverse-flow approach.}
	We briefly comment on the inverse-flow strategy under reduced regularity assumptions, with particular reference to \cite{OuknineTurpin2006}.
	Broadly speaking, the inverse-flow strategy proceeds by identifying an SDE satisfied by	the inverse stochastic flow and then estimating quantities such as the Jacobian determinant	of the inverse flow that arise in change-of-variable arguments.
	
	In \cite{OuknineTurpin2006}, the inverse-flow SDE is recalled after using Kunita's theory \cite{Kunita1984} to obtain a stochastic flow of homeomorphisms under Lipschitz coefficients,	while the additional assumptions are stated mainly as a Lipschitz condition on the quantity $\sigma^*\nabla\sigma$ (in their notation).
	From our standpoint, it is not immediate that these assumptions alone suffice to justify the same inverse-flow SDE beyond the smooth setting.
	Since the subsequent proof relies on an SDE on a product space based on this inverse-flow representation, making this step explicit is therefore essential for a fully rigorous treatment in such a low-regularity framework.
	
	Finally, although we will not pursue this point in detail here, change-of-variables arguments for weakly differentiable homeomorphisms also call for some care, in particular in dimensions $d>2$.
	In our setting this can be handled by strengthening the relevant $L^2$-type controls to $L^p$-type estimates with $p>d$, ensuring the needed Lusin property.
	
	These observations motivate our self-contained analysis of the inverse-flow SDE under reduced regularity (Proposition~\ref{th:inverse_flow_SDE}), and our quantitative control of the associated Jacobian terms, which form key ingredients in the proof of the main norm equivalence theorem. 

	\medskip	
	\noindent
	\textbf{Main results and contributions.}
	The present paper revisits the inverse-flow approach and establishes a strengthened and more flexible norm equivalence theorem for SDEs with time-dependent coefficients.
	Our main contributions are summarized as follows.
	
	\begin{enumerate}[label={(\roman*)},wide,topsep=3.1pt,itemsep=-2.4pt]
		\item \emph{Locally Lipschitz coefficients.}
		We establish norm equivalence \eqref{eq:intro.norm_eq_result1}--\eqref{eq:intro.norm_eq_result2} for SDEs with (possibly time-dependent) \emph{locally} Lipschitz coefficients.
		To the best of our knowledge, such a locally Lipschitz version has not been explicitly available in the existing norm equivalence literature.
		
		\item \emph{Weaker time regularity and explicit constants.}
		Our assumptions allow the relevant growth and Lipschitz controls to be governed by integrable functions of time.
		Moreover, we obtain an explicit dependence of the equivalence constants: in Theorem~\ref{th:main_theorem_locally_Lipschitz} we may take $c=e^{-5\|\widetilde{K}\|_{L^1([0,T])}}$ and $C=e^{5\|\widetilde{K}\|_{L^1([0,T])}}$.
		This quantitative form clarifies the stability of the principle and is particularly useful for extensions to long time horizons, including the infinite-horizon case.
		
		\item \emph{Exponential weights and the role of boundedness.}
		In the globally Lipschitz setting, we clarify the role of boundedness of the coefficients for exponential weights by providing explicit counterexamples (Remark~\ref{rem:counterexamples}).
		Under this boundedness condition, we treat exponential weights $\rho=e^{F}$ under assumptions comparable to those in \cite{BallyMatoussi2001}; see Corollary~\ref{th:cor_of_main_theorem}.
		In particular, this shows that the additional integrability requirements on $\rho$ imposed in later formulations such as \cite{OuknineTurpin2006} are not needed in the present framework.
				
		\item \emph{Inverse flow SDE under reduced smoothness.}
		As a core tool for the main theorem, we show that the inverse flow satisfies an It\^o equation
		involving a backward It\^o integral together with an explicit correction term under reduced spatial regularity of $\sigma$ and weaker time assumptions (Proposition~\ref{th:inverse_flow_SDE}).
		This result provides a rigorous foundation for the inverse-flow strategy in our setting and may be of independent interest.
	\end{enumerate}
	
	\medskip
	\noindent
	\textbf{Ideas of the proof.}
	At a conceptual level, our approach is inspired by the stochastic flow/Jacobian method of Bally--Matoussi~\cite{BallyMatoussi2001} and by the weak-derivative viewpoint highlighted in Ouknine--Turpin~\cite{OuknineTurpin2006}.
	In contrast to the arguments in \cite{OuknineTurpin2006} that rely on results from Dirichlet form theory, we prepare the required weak differentiability properties of SDE solutions by a different and more direct approach.
	Moreover, we combine these ingredients with a stochastic Liouville-type representation for the Jacobian determinant.
	This allows us to formulate the key assumptions in terms of quantities such as $\tr\nabla b$ and $\tr \nabla\sigma_{(k)}$, thereby leading to conditions that remain verifiable under weaker hypotheses (for instance, bounded divergence of the drift) beyond the globally Lipschitz regime.
	
	Two technical points deserve emphasis.
	First, to justify the inverse-flow SDE in our low-regularity setting, we prove a stability result for mollified approximations adapted to the present assumptions (Lemma~\ref{th:approximation_molifier_1}).
	Second, the extension from global to local Lipschitz coefficients is achieved by a carefully designed approximation procedure, where an additional argument (Step~3 in the proof of
	Theorem~\ref{th:main_theorem_locally_Lipschitz}) is needed to retain uniform weighted estimates.
	
	\medskip
	\noindent
	\textbf{Organization of the paper.}
	The remainder of the paper is organized as follows.
	Section~\ref{section:Main_theorem} introduces the setting and states the main theorem and its corollaries.
	Section~\ref{section:inverse_flow_main} establishes the inverse-flow SDE with a backward It\^o integral under reduced regularity.
	Section~\ref{section:weak_diffrentability_of_sol_of_SDE} presents weak differentiability results for solutions to SDEs and derives the Jacobian estimates required in the proof.
	Section~\ref{section:proof_of_main_theorem} contains the proofs of the main results.
	Finally, Section~\ref{section:application} presents an application of norm equivalence to weighted integrability estimates for solutions of the associated PDEs.

	\section{Main results}\label{section:Main_theorem}
	For a matrix $A$, we write $|A|$ for its Frobenius norm and $|A|_{\mathrm{op}}$ for its operator norm.
	
	Let $T\in(0,\infty)$ and let $(\Omega,\mathscr F,(\mathscr F_t)_{t\in[0,T]},\Prob)$ be a filtered probability space satisfying the usual conditions. Let $(W_t)_{t\in[0,T]}$ be a $d'$-dimensional $(\mathscr F_t)$-Brownian motion on this space.
	
	Throughout this section, let $b:[0,T]\times\R^d\to\R^d$ and $\sigma:[0,T]\times\R^d\to\R^{d\times d'}$ be Borel measurable functions. Consider the following assumptions:	
	\begin{enumerate}[label={(A\arabic*)}]
		\item \label{item:asummption_of_main_theorem_2}
		There exists a nonnegative function $K\in L^1([0,T])$ such that
		\begin{align*}
			|b(r,x)|\leq K(r)(1+|x|), \quad
			|\sigma(r,x)|\leq K(r)^{1/2}(1+|x|),\quad x\in\R^d,\;\; r\in[0,T].
		\end{align*}
		
		\item \label{item:asummption_of_main_theorem_3}
		For every $N=1,2,\dots$, there exists a nonnegative function $K_N\in L^1([0,T])$ such that
		\begin{align*}
			&|b(r,x)-b(r,x')| \leq K_N(r)|x-x'|,\\
			&|\sigma(r,x)-\sigma(r,x')| \leq K_N(r)^{1/2}|x-x'|,\quad |x|,|x'|\leq N,\;\; r\in[0,T].
		\end{align*}
		
		\item \label{item:asummption_of_main_theorem_1}
		There exists a Borel measurable function $\widehat{\sigma}:[0,T]\times\R^d\to\R^d$ such that, for each $r$, the map $x\mapsto\widehat{\sigma}(r,x)$ is locally Lipschitz continuous and
		\begin{align}\label{eq:asummption_of_main_theorem_1}
			\widehat{\sigma}_i(r,x)=\sum_{\substack{1\leq j \leq d\\ 1\leq k \leq d'}}\sigma_{jk}(r,x)\partial_j\sigma_{ik}(r,x),\quad \text{a.e.}\;x,\;\; i=1,\dots,d.
		\end{align}
		 Here, for each $r$, $\partial_j\sigma_{ik}(r,x)$ denotes an a.e.-defined partial derivative of the locally Lipschitz function $x\mapsto\sigma_{ik}(r,x)$.
		
		\item \label{item:asummption_of_main_theorem_4}
		For every $N=1,2,\dots$, there exists a nonnegative function $\widehat{K}_N\in L^1([0,T])$ with $\lim_{N\to\infty}\|\widehat{K}_N\|_{L^1}/\log N=0$ such that
		\[
		|\widehat{\sigma}(r,x)| \leq \widehat{K}_N(r)(1+|x|),\quad |x|\leq N,\;\; r\in[0,T].
		\]
	\end{enumerate}
	
	Under these assumptions, for each $(t,x)\in[0,T]\times\R^d$, let $(X_s^{t,x})_{s\in[t,T]}$ be the solution to the SDE
	\begin{align}\label{eq:SDE}
		X_s^{t,x}=x+\int_t^s b(r,X_r^{t,x})\,dr+\int_t^s \sigma(r,X_r^{t,x})\,dW_r,\quad s\in[t,T].
	\end{align}
	For existence and uniqueness of solutions to such SDEs with (possibly time-dependent) locally Lipschitz coefficients, we refer the reader to, for example, \cite[Proposition~3.28]{PardouxRascanu2014} and \cite[Theorem~119]{Situ2005}.
	
	The following theorem is the main result of this paper.
	
	\begin{thm}[Main theorem]\label{th:main_theorem_locally_Lipschitz}
		Assume \ref{item:asummption_of_main_theorem_2}--\ref{item:asummption_of_main_theorem_4}.
		Let $\rho:\R^d\to(0,\infty)$ be a $C^2$ function, and suppose that there exists a nonnegative function $\widetilde{K}\in L^1([0,T])$ such that
		\begin{align}
			\begin{aligned}\label{eq:main_theorem_locally_Lipschitz'}
				&|\langle b(r,x),\nabla\rho(x)\rangle|+|\langle \widehat{\sigma}(r,x),\nabla\rho(x)\rangle|\leq \widetilde{K}(r)\rho(x),\\
				&|\sigma(r,x)||\nabla\rho(x)|\leq \widetilde{K}(r)^{1/2}\rho(x),\quad
				|\sigma(r,x)|^2|\nabla^2\rho(x)|_{\mathrm{op}}\leq \widetilde{K}(r)\rho(x)
			\end{aligned}
		\end{align}
		for every $x\in\R^d$ and $r\in[0,T]$, and such that
		\begin{align}
			\begin{aligned}\label{eq:main_theorem_locally_Lipschitz}
				&\bigg|\tr\nabla b(r,x)-\tr\nabla\widehat{\sigma}(r,x)+\frac12\sum_{k=1}^{d'}\tr\big[(\nabla\sigma_{(k)}(r,x))^2\big]\bigg|\leq \widetilde{K}(r),\\
				&\sum_{k=1}^{d'}\big(\!\tr\nabla\sigma_{(k)}(r,x)\big)^2\leq \widetilde{K}(r),\quad \text{a.e.}\; x\in\R^d,\;\; \text{a.e.}\; r\in[0,T],
			\end{aligned}
		\end{align}
		where $\sigma_{(k)}$ denotes the $k$-th column of $\sigma$. Then, for every Borel measurable function $\varphi:\R^d\to\R$ and $\psi:[0,T]\times\R^d\to\R$,
		\begin{align}
			&c\int_{\R^d}|\varphi(x)|\rho(x)\,dx\leq\int_{\R^d}\E[|\varphi(X^{t,x}_s)|]\rho(x)\,dx\leq C\int_{\R^d}|\varphi(x)|\rho(x)\,dx,\quad 0\leq t\leq s\leq T,\label{eq:norm_eq_result1}\\
			&\begin{aligned}
				c\int_{\R^d}\int_t^T|\psi(s,x)|\rho(x)\,ds\,dx
				&\leq\int_{\R^d}\int_t^T\E[|\psi(s,X^{t,x}_s)|]\rho(x)\,ds\,dx\\
				&\qquad\qquad\leq C\int_{\R^d}\int_t^T|\psi(s,x)|\rho(x)\,ds\,dx,\quad t\in[0,T].
			\end{aligned}\label{eq:norm_eq_result2}
		\end{align}
		Here $c=e^{-5\|\widetilde{K}\|_{L^1}}$ and $C=e^{5\|\widetilde{K}\|_{L^1}}$.
	\end{thm}
	
	In the globally Lipschitz case, and for weights $\rho$ of the following forms, we obtain the following corollary.
	This result extends \cite[Proposition~5.1]{BallyMatoussi2001} and \cite[Proposition~3.1]{OuknineTurpin2006} to the case where the Lipschitz constants are allowed to depend on time.
	\begin{cor}\label{th:cor_of_main_theorem}
		Assume \ref{item:asummption_of_main_theorem_2}--\ref{item:asummption_of_main_theorem_1} and, in addition, that for the same $K\in L^1([0,T])$ as in \ref{item:asummption_of_main_theorem_2},
		\begin{align}\label{eq:cor_of_main_theorem02}
			\begin{aligned}
				&|b(r,x)-b(r,x')|\leq K(r)|x-x'|,\quad
			|\sigma(r,x)-\sigma(r,x')|\leq K(r)^{1/2}|x-x'|,\\
			&|\widehat{\sigma}(r,x)-\widehat{\sigma}(r,x')|\leq K(r)|x-x'|,\qquad x,x'\in\R^d,\;\; r\in[0,T].
			\end{aligned}
		\end{align}
		If $\rho:\R^d\to(0,\infty)$ is one of the following, then there exist constants $c,C>0$ such that, for every Borel measurable $\varphi:\R^d\to\R$ and $\psi:[0,T]\times\R^d\to\R$, the bounds \eqref{eq:norm_eq_result1} and \eqref{eq:norm_eq_result2} hold:
		\begin{enumerate}
			\item For $\beta\in\R$, let $\rho(x)=(1+|x|^2)^{\beta}$. In this case $c$ and $C$ depend only on $d,\beta,$ and~$\|K\|_{L^1}$.
			\item Suppose, in addition, that
			\begin{align}\label{eq:cor_of_main_theorem01}
				|b(r,x)|\leq K(r),\quad |\sigma(r,x)|\leq K(r)^{1/2},\quad x\in\R^d,\;\; r\in[0,T].
			\end{align}
			Then we may take $\rho=\exp F$, where $ F:\R^d \to \R $ is a locally bounded Borel function which is $ C^2 $ on $ \{x\in\R^d \mid |x|>R_0 \} $ with bounded derivatives there, for some $ R_0>0 $.
		\end{enumerate}
	\end{cor}
	
	\begin{rem}\label{rem:counterexamples}
		In (2) above, if one drops the boundedness condition \eqref{eq:cor_of_main_theorem01}, norm equivalence for exponential weights does not hold in general.
		Two simple counterexamples are the geometric Brownian motion and the Ornstein--Uhlenbeck process.
		
		To see this, we work in one dimension.  
		First, let $\alpha\neq 0$ and $\beta\in\R$, and consider the SDE
		\[
		dX_s^x = \beta X_s^x\,ds + \alpha X_s^x\,dW_s,\quad X_0^x=x.
		\]
		Since the solution is given by $ X_s^x = x \exp\{(\beta-\alpha^2/2)s + \alpha(W_s-W_0)\} $, we have $\E[e^{X_s^x}]=\infty$ for every $x>0$ whenever $s>0$. 
		Consequently,
		\[
		\int_{\R}\E[e^{X_s^x}]e^{-2|x|}\,dx = \infty,
		\quad\text{whereas}\quad
		\int_{\R} e^x e^{-2|x|}\,dx < \infty.
		\]
		
		Next, let $\lambda>0$, and consider the Ornstein--Uhlenbeck SDE
		\[
		dX_s^x = -\lambda X_s^x\,ds + dW_s,\quad X_0^x = x.
		\]
		For $s>0$, the solution $X_s^x$ is Gaussian with mean $x e^{-\lambda s}$ and variance
		$(1-e^{-2\lambda s})/(2\lambda)$.  Let $p(s,x,y)$ denote its transition density.  
		Then for any Borel function $\varphi$ and any $a\in\R$,
		\[
		\int_{\R}\E[|\varphi(X_s^x)|] e^{ax}\,dx
		= \int_{\R} |\varphi(y)| 
		\bigg( \int_{\R} p(s,x,y) e^{ax}\,dx \bigg)\,dy.
		\]
		Thus, for norm equivalence to hold, there must exist constants $c,C>0$ such that
		\[
		c e^{ay}
		\le \int_{\R} p(s,x,y) e^{ax}\,dx
		\le C e^{ay},
		\quad y\in\R.
		\]
		However, this fails unless $a=0$.
	\end{rem}
	
	\begin{rem}\label{rem:generalization_of_cor_1}
		Corollary~\ref{th:cor_of_main_theorem} can be extended to the generalized equivalence of norm principle in \cite{ZhangZhao2007}, in which the functions $\varphi$ and $\psi$ are allowed to be random.
		We postpone the details of this extension to Remark~\ref{rem:generalization_of_cor_2}, which appears after the proof of Corollary~\ref{th:cor_of_main_theorem}.
	\end{rem}

	\section{Inverse flow}\label{section:inverse_flow_main}
	This section is devoted to the study of the inverse stochastic flow, which, under suitable assumptions, satisfies a stochastic differential equation. 
	This representation forms one of the key tools in the proof of the main theorem.
	
	\subsection{Approximation by mollifiers}
	First, we provide stability properties of the space-time mollification, which are used in the argument for the inverse flow (see Subsection~\ref{section:inverse_flow}).
	To this end, we begin with the following lemma.
	
	\begin{lem}\label{th:continuity_of_translation_in_Lp}
		Let $ f:\R \times \R^d \to \R^{d_0} $ be Borel measurable and assume that for each $ r\in \R $, the map $ x\mapsto f(r,x) $ is continuous. Let $ p\geq 1 $ and suppose there exists a nonnegative function $ K\in L^p(\R) $ such that
		\begin{equation}\label{eq:continuity_of_translation_in_Lp}
			|f(r,x)|
			\leq K(r)(1+ |x|),\quad x\in \R^d,\;\; r\in \R			
		\end{equation}
		Then, for any $ R>0 $,
		\[
		\lim_{s\to 0,\,y\to 0}\int_0^T \sup_{|x|\leq R}|f(r-s,x-y)-f(r,x)|^p \,dr=0.
		\]
	\end{lem}
	
	\begin{proof}
		It suffices to consider the case $ d_0=1 $. Let $ s\in\R,\,y\in \R^d $ with $ |s|\leq 1,\, |y|\leq 1 $. Observe that
		\begin{align*}
			&\int_0^T \sup_{|x|\leq R}|f(r-s,x-y)-f(r,x)|^p \,dr\\
			&\leq 2^p\int_0^T \sup_{|x|\leq R}|f(r-s,x-y)-f(r,x-y)|^p \,dr + 2^p\int_0^T \sup_{|x|\leq R}|f(r,x-y)-f(r,x)|^p \,dr\\
			&=: 2^p(I_1(s,y) + I_2(y)) .
		\end{align*}
		First, for fixed $ r\in \R $, the map $ x\mapsto f(r,x) $ is continuous, hence uniformly continuous on $ \{x\in\R^d\mid |x|\le R+1\} $. Therefore, by \eqref{eq:continuity_of_translation_in_Lp} and the dominated convergence theorem, $ \lim_{y\to 0}I_2(y)=0 $.
		
		We next show $ \lim_{s\to 0,\,y\to 0}I_1(s,y)=0 $. Let $ B_{R+1}:=\{x\in \R^d \mid |x|\leq R+1 \} $ and define $ F:\R\to C(B_{R+1}) $ by $ F(r) := f(r,\cdot\,)|_{B_{R+1}}$ for $r \in \R$.
		The space $ C(B_{R+1}) $ is a separable Banach space under the supremum norm $ \|\cdot\|_{C(B_{R+1})} $, and for every $ g \in C(B_{R+1}) $, the map $r\mapsto \|F(r)-g\|_{C(B_{R+1})} = \sup_{|x|\leq R+1}|f(r,x)-g(x)|$ is Borel measurable; hence $ F $ is a Borel measurable function with separable range.
		Consequently, by \eqref{eq:continuity_of_translation_in_Lp} we have $ F\in L^p(\R; C(B_{R+1})) $. Therefore, by the continuity of translations in Bochner $ L^p $-spaces (see, e.g., \cite[Section 1.1]{Arendt_etal.2011}),
		\begin{align*}
			\sup_{|y|\leq 1} I_1(s,y)
			&\leq  \int_{\R} \|F(r-s)-F(r) \|_{C(B_{R+1})}^p\,dr
			\to 0\quad (s\to 0).
		\end{align*}
		Hence $ \lim_{s\to 0,\,y\to 0}I_1(s,y)=0 $. This completes the proof.
	\end{proof}
	
	By a sequence of mollifiers on $ \R^{d_1} $, we mean a sequence of nonnegative functions $ (\rho_n)_{n\geq 1} $ on $ \R^{d_1} $ such that
	\[
	\rho_n\in C_c^\infty(\R^{d_1}),\quad
	\supp \rho_n \subset \{z\in \R^{d_1}\mid |z|\leq 1/n \},\quad
	\int_{\R^{d_1}} \rho_n(z)\,dz=1 .
	\]
	
	Using the previous lemma, we now establish the following stability result for space-time mollification. (These will be used later in the proof of Lemma~\ref{th:inverse_flow_stochastic_integral}.)
	
	\begin{lem}\label{th:approximation_molifier_1}
		Let $ (\rho_n)_{n\geq 1} $ and $ (\rho_n^1)_{n\geq 1} $ be sequences of mollifiers on $ \R^d $ and on $ \R $, respectively.
		\begin{enumerate}
			\item Let $ f:\R \times \R^d \to \R^{d_0} $ be Borel measurable, and assume that for each $ r\in \R $, the map $ x\mapsto f(r,x) $ is continuous. Let $ p\geq 1 $ and suppose there exists a nonnegative function $ K\in L^p(\R) $ such that
			\[
			|f(r,x)|
			\leq K(r)(1+ |x|),\quad x\in \R^d,\;\; r\in \R .
			\]
			Define
			\[
			f_n(r,x) := \int_{\R}\int_{\R^d} f(r-s,x-y) \rho_n(y)\rho_n^1(s)\,dy\,ds, \quad (r,x)\in \R\times \R^d .
			\]
			Then, for any $ R>0 $,
			\[
			\lim_{n\to \infty}\int_0^T \sup_{|x|\leq R}|f_n(r,x)-f(r,x)|^p \,dr=0 .
			\]			
			\item Let $ f,\,g:\R \times \R^d \to \R^{d_0} $ and $ F:\R \times \R^d \to \R $ be Borel measurable. Assume that for each $ r\in \R $, the maps $ x\mapsto f(r,x) $ and $ x\mapsto F(r,x) $ are continuous, and
			\[
			F(r,x)= \langle f(r,x), g(r,x) \rangle,\quad \text{a.e.}\;x\in\R^d .
			\]
			Suppose in addition that there exists a nonnegative function $ K\in L^1(\R) $ such that
			\[
			|f(r,x)|\leq K(r)(1+ |x|), \quad |g(r,x)|\leq K(r)^{1/2},\quad  x\in \R^d,\;\;  r\in \R .
			\]
			For $ (r,x)\in \R\times \R^d $, set
			\begin{align*}
				&g_n(r,x) := \int_{\R}\int_{\R^d} g(r-s,x-y) \rho_n(y)\rho_n^1(s)\,dy\,ds
			\end{align*}
			and $F_n(r,x):= \langle f(r,x), g_n(r,x) \rangle$. Then, for any $ R>0 $,
			\[
			\lim_{n\to \infty}\int_0^T \sup_{|x|\leq R}|F_n(r,x)-F(r,x)| \,dr=0.
			\]
		\end{enumerate}
	\end{lem}
	
	\begin{proof}
		(1) Since $ \int_{\R} \rho_n^1(s)\,ds= 1 $ and $ \int_{\R^d}\rho_n(y)\,dy = 1 $, H\"older's inequality yields
		\begin{align*}
			|f_n(r,x)-f(r,x)|^p
			&\leq \int_{\R}\int_{\R^d} |f(r-s,x-y)-f(r,x)|^p \,\rho_n(y)\rho_n^1(s)\,dy\,ds .
		\end{align*}
		Hence, by Lemma~\ref{th:continuity_of_translation_in_Lp},
		\begin{align*}
			\int_0^T \sup_{|x|\leq R}|f_n(r,x)-f(r,x)|^p \,dr
			&\leq  \sup_{\substack{|s|\leq 1/n\\ |y|\leq 1/n}}\int_0^T\sup_{|x|\leq R}|f(r-s,x-y)-f(r,x)|^p\,dr
			\to 0
		\end{align*}
		as $n\to\infty$.
		
		(2) Define
		\[
		\widetilde{F}_n(r,x)
		:= \int_{\R}\int_{\R^d} F(r-s,x-y) \rho_n(y)\rho_n^1(s)\,dy\,ds, \quad (r,x)\in \R\times \R^d .
		\]
		Then
		\begin{align*}
			&\int_0^T \sup_{|x|\leq R}|F_n(r,x)-F(r,x)|\,dr\\
			&\leq \int_0^T \sup_{|x|\leq R} |F_n(r,x)-\widetilde{F}_n(r,x)|\,dr + \int_0^T \sup_{|x|\leq R}|\widetilde{F}_n(r,x)-F(r,x)|\,dr
			=: I^1_n + I^2_n.
		\end{align*}
		By the assumptions, for each $r\in\R$ we have
		\begin{align*}
			|F(r,x)|
			\leq  |f(r,x)||g(r,x)|
			\leq K(r)(1+ |x|)\quad \text{for a.e.}\;x\in \R^d,
		\end{align*}
		and by the continuity of $x\mapsto F(r,x)$ the same bound holds for all $x$. Hence, by part (1), $ I^2_n\to 0 $ as $ n\to\infty $.
		To estimate $ I^1_n $, note that
		\begin{align*}
			F_n(r,x)
			&= \int_{\R}\int_{\R^d} \langle f(r,x), g(r-s,x-y) \rangle \rho_n(y)\rho_n^1(s)\,dy\,ds,\\
			\widetilde{F}_n(r,x)
			&= \int_{\R}\int_{\R^d} \langle f(r-s,x-y), g(r-s,x-y) \rangle \rho_n(y)\rho_n^1(s)\,dy\,ds,
		\end{align*}
		so by the Schwarz inequality,
		\begin{align*}
			|F_n(r,x)-\widetilde{F}_n(r,x)|
			\begin{multlined}[t]
				\leq  \bigg(\int_{\R}\int_{\R^d} |f(r,x)- f(r-s,x-y)|^2 \rho_n(y)\rho_n^1(s)\,dy\,ds\bigg)^{1/2}\qquad\\
				\bigg(\int_{\R}\int_{\R^d} |g(r-s,x-y)|^2\rho_n(y)\rho_n^1(s)\,dy\,ds \bigg)^{1/2}.
			\end{multlined}
		\end{align*}
		Using again Schwarz together with $ |g(r-s,x-y)|\leq K(r-s)^{1/2} $,
		\begin{align*}
			I^1_n
			&\begin{multlined}[t]
				\leq \bigg(\int_0^T \sup_{|x|\leq R}\int_{\R}\int_{\R^d} |f(r,x)- f(r-s,x-y)|^2 \rho_n(y)\rho_n^1(s)\,dy\,ds\,dr\bigg)^{1/2} \\
				\bigg(\int_0^T\int_{\R}\int_{\R^d}K(r-s)\rho_n(y)\rho_n^1(s)\,dy\,ds\,dr \bigg)^{1/2}
			\end{multlined}\\
			&\leq \|K\|_{L^1}^{1/2}\bigg(\sup_{\substack{|s|\leq 1/n\\ |y|\leq 1/n}}\int_0^T\sup_{|x|\leq R}|f(r-s,x-y)-f(r,x)|^2\,dr\bigg)^{1/2}.
		\end{align*}
		By Lemma~\ref{th:continuity_of_translation_in_Lp} the right-hand side tends to $ 0 $ as $ n\to\infty $. This completes the proof.
	\end{proof}
	
	\subsection{Backward It\^o integral}
	We recall the backward It\^o integral (see, e.g., \cite{Kunita1984,Kunita1990} and \cite{Kunita2019}). In this subsection fix $s\in(0,T]$. Then
	\[
	\widetilde{W}_r:=W_s-W_{s-r},\quad r\in[0,s],
	\]
	is a $d'$-dimensional Brownian motion. Set
	\[
	\mathscr{G}_r:= \sigma(\,\widetilde{W}_t \mid t\in[0,r]\,)\vee \mathscr{N},\quad r\in[0,s],
	\]
	where $\mathscr{N}$ denotes the set of  $\Prob $-null sets.
	
	\begin{dfn}\label{dfn:backward_Ito_integral}
		Let $(f_r)_{r\in[0,s]}$ be a real-valued process such that $(f_{s-r})_{r\in[0,s]}$ is $(\mathscr{G}_r)_{r\in[0,s]}$-progressively measurable and $\int_{0}^{s} |f_r|^2\,dr < \infty\;\text{a.s.}$
		For $k=1,\dots,d'$, define
		\[
		\int_{t}^{s} f_r\,\widehat{d}W_r^k := \int_{0}^{s-t} f_{s-r}\,d\widetilde{W}_r^k,\quad t\in[0,s],
		\]
		and call this the backward It\^o integral. Here the right-hand side is the usual stochastic integral with respect to the $(\mathscr{G}_r)_{r\in[0,s]}$-Brownian motion.
	\end{dfn}

	\begin{rem}\label{rem:backward_Ito_integral}
		The above definition of the backward It\^o integral is equivalent to that in \cite{Kunita1984,Kunita1990}. Indeed, if $(f_r)_{r\in[0,s]}$ is continuous and satisfies the hypotheses of Definition~\ref{dfn:backward_Ito_integral}, then for a partition $\Delta: t=r_0<r_1<\cdots <r_n=s$ of $[t,s]$, we have convergence in probability
		\[
		\int_{t}^{s} f_r\,\widehat{d}W_r^k
		= \lim_{|\Delta|\to 0}\sum_{j=0}^{n-1}f_{r_{j+1}}(W_{r_{j+1}}^k - W_{r_{j}}^k).
		\]
		To see this, set $u_j:=s-r_{n-j}$; then
		\begin{align*}
			\sum_{j=0}^{n-1}f_{r_{j+1}}(W_{r_{j+1}}^k-W_{r_{j}}^k)
			&= \sum_{j=0}^{n-1}f_{s-u_j}(W_{s-u_j}^k-W_{s-u_{j+1}}^k)
			= \sum_{j=0}^{n-1}f_{s-u_j}(\widetilde{W}_{u_{j+1}}^k-\widetilde{W}_{u_j}^k),
		\end{align*}
		from which the claim follows.
	\end{rem}

	\subsection{Inverse flow}\label{section:inverse_flow}
	We now discuss the SDE and the (inverse) flow in our setting. Throughout this subsection, we work under the following standing assumptions:
	
	\begin{assum}\label{assum:globally_lipschitz_coefficients}
		$ b: [0,T] \times \R^d\to \R^d $ and $ \sigma: [0,T] \times \R^d\to \R^{d\times d'} $ are Borel measurable, and there exists a nonnegative function $ K\in L^1([0,T]) $ such that
		\begin{align*}
			&|b(r,x)-b(r,x')| \leq K(r)|x-x'|,\quad
			|b(r,x)| \leq K(r)(1+|x|),\notag\\
			&|\sigma(r,x)-\sigma(r,x')| \leq K(r)^{1/2}|x-x'|,\quad
			|\sigma(r,x)| \leq K(r)^{1/2}(1+|x|)
		\end{align*}
		for all $x,x'\in\R^d$ and $r\in[0,T]$.
	\end{assum}
	
	Under this assumption, for each $ (t,x) \in [0,T]\times\R^d  $ let $ (X_s^{t,x})_{s\in[t,T]} $ denote the solution to the SDE \eqref{eq:SDE}. We have the following standard estimates.
	
	\begin{lem}\label{th:SDE_standard_estim}
		Let Assumption~\ref{assum:globally_lipschitz_coefficients} hold. Then for any $ p\geq 2 $, there exists a constant $ C_p>0 $ depending only on $ p $ such that for all $ t\in[0,T] $ and $ x,x'\in \R^d $,
		\begin{align*}
			&\E\bigg [\sup_{ s \in[t,T]}|X_s^{t,x}|^p\bigg ]
			\leq C_p(1+|x|^p),\quad
			\E\bigg [\sup_{ s \in[t,T]}|X_s^{t,x}-X_s^{t,x'}|^p\bigg ]
			\leq C_p|x-x'|^p.
		\end{align*}
	\end{lem}
	
	\begin{proof}
		This follows from Lemmas~\ref{th:SDE_apriori1} and \ref{th:SDE_apriori_diff}.
	\end{proof}
	
	We also provide an estimate concerning the approximation of quadratic covariations.
	
	\begin{lem}\label{th:estim_of_approximation_of_quadratic_varitaions_for_g(r,X)}
		Suppose that Assumption~\ref{assum:globally_lipschitz_coefficients} holds.
		Let $g:[0,T]\times\R^d\to\R$ be a Borel measurable function such that, for each $r$, the map $x\mapsto g(r,x)$ is of class $C^1$, and there exists a constant $L>0$ with
		\begin{align}\label{eq:estim_of_approximation_of_quadratic_varitaions_for_g(r,X)_2}
			|\nabla g(r,x) - \nabla g(r',x')|
			\leq L(|r-r'| + |x-x'|),\quad
			|\nabla g(r,x)|\leq L .
		\end{align}
		Let $ 0\leq t<s\leq T $ and let $ \Delta: t= r_0<r_1<\cdots < r_n=s $ be a partition of $ [t,s] $. For $ k=1,\dots,d' $, set
		\[
		V_k^{\Delta,t,x}:= \sum_{j=0}^{n-1}\big(g(r_{j+1},X_{r_{j+1}}^{t,x})- g(r_{j},X_{r_{j}}^{t,x})\big)\big(W_{r_{j+1}}^k- W_{r_{j}}^k\big).
		\]
		Then for any $ p\geq 1 $, there exists a constant $ C_p> 0 $ independent of $ \Delta $ such that
		\[
		\E\big[|V_k^{\Delta,t,x}- V_k^{\Delta,t,x'}|^p\big]
		\leq C_p(1+|x| +|x'|)^{p} |x-x'|^p,\quad x,x'\in \R^d.
		\]
	\end{lem}
	
	\begin{proof}
		See Appendix~\ref{section:proof_of_estim_of_approximation_of_quadratic_varitaions_for_g(r,X)}.
	\end{proof}
	
	Next, we recall the flow associated with the SDE. According to Theorem~4.5.1 of \cite{Kunita1990}, the family $(X_s^{t,x})_{0\le t\le s\le T,\;x\in\R^d}$ admits a modification which defines a stochastic flow of homeomorphisms. 
	We continue to denote this modification by the same notation $(X_s^{t,x})_{0\le t\le s\le T,\;x\in\R^d}$. In particular, the following properties hold with probability one:
	\begin{enumerate}
		\item the map $(t,s,x)\mapsto X_s^{t,x}$ is continuous;
		\item for any $0\le t\le s\le T$, the map $x\mapsto X_s^{t,x}$ is a homeomorphism of $\R^d$;
		\item for any $0\le t\le r\le s\le T$ and $x\in\R^d$, we have the flow property
		\[
		X_s^{r,y}\big|_{y=X_r^{t,x}}=X_s^{t,x}.
		\]
	\end{enumerate}
	For each $t,s$, we denote the inverse of the map $x\mapsto X_s^{t,x}$ as $x\mapsto \widehat{X}_s^{t,x}$, and refer to it as the inverse flow. 
	For each $x\in\R^d$, the random variable $\widehat{X}_s^{t,x}$ is measurable with respect to the $\sigma$-algebra generated by $\{W_s-W_r\mid r\in[t,s]\}$ together with all $\Prob $-null sets. Moreover, by the flow property, for every $r\in[t,s]$ we have
	\[
	X_r^{t,y}\big|_{y=\widehat{X}_s^{t,x}}=\widehat{X}_s^{r,x}
	\]
	and from (1), (2) it follows that $(r,x)\mapsto \widehat{X}_s^{r,x}$ is almost surely continuous. In what follows we also use the shorthand $X_s^{t}(x):=X_s^{t,x}$. 
	
	With these preliminaries in place, we proceed to show that the inverse flow satisfies an It\^o equation involving the backward It\^o integral. To this end, we prove the following lemma, an analogue of \cite[Chapter~II, Lemma~6.2]{Kunita1984}.
	
	\begin{lem}\label{th:inverse_flow_stochastic_integral}
		Let Assumption~\ref{assum:globally_lipschitz_coefficients} hold. Let $g:[0,T]\times\R^d\to\R^{d\times d'}$ be Borel measurable and satisfy
		\begin{align}\label{eq:inverse_flow_stochastic_integral_3}
			|g(r,x)-g(r,x')|\le K(r)^{1/2}|x-x'|,\quad
			|g(r,x)|\le K(r)^{1/2}(1+|x|),
		\end{align}
		and suppose there exists a Borel measurable function $\widehat{g}:[0,T]\times\R^d\to\R^d$ such that, for each $r$, the map $x\mapsto \widehat{g}(r,x)$ is continuous and
		\begin{align}\label{eq:inverse_flow_stochastic_integral_1}
			\widehat{g}_i(r,x)
			= \sum_{\substack{1\le j\le d\\ 1\le k\le d'}} \sigma_{jk}(r,x)\,\partial_j g_{ik}(r,x),
			\quad \text{for a.e.}\;x,\;\; i=1,\dots,d .
		\end{align}
		Then, for fixed $0\le t\le s\le T$, the family $\bigl(\int_t^s g(r,X_r^{t,x})\,dW_r\bigr)_{x\in\R^d}$ admits a modification that is continuous in $x$, which we denote by $\bigl(\overline{\int_t^{s}g(r,X_r^{t,x})\,dW_r}\bigr)_{x\in\R^d}$. Moreover, for every $y\in\R^d$, almost surely
		\begin{align}\label{eq:inverse_flow_stochastic_integral_2}
			\overline{\int_{t}^{s}g(r,X_r^{t,x})\,dW_r}\bigg|_{x=\widehat{X}_s^{t,y}}
			= \int_{t}^{s} g(r,\widehat{X}_s^{r,y})\,\widehat{d}W_r
			- \int_{t}^{s}\widehat{g}(r,\widehat{X}_s^{r,y})\,dr .
		\end{align}
	\end{lem}
	
	\begin{proof}
		\textbf{Step 0.}
		We start by showing that the stochastic integral admits a continuous modification. By the Burkholder--Davis--Gundy (BDG) inequality, for any $ p\geq 1 $ there exists a constant $ c_p>0 $ such that
		\begin{align*}
			&\E\bigg[\bigg|\int_{t}^{s}g(r,X_r^{t,x})\,dW_r- \int_{t}^{s}g(r,X_r^{t,x'})\,dW_r\bigg|^p\bigg]\\
			&\leq c_p \E\bigg[\bigg(\int_{t}^{s}|g(r,X_r^{t,x})-g(r,X_r^{t,x'})|^2\,dr\bigg)^{p/2}\bigg]
			\leq c_p\|K\|_{L^1}^{p/2}\E\bigg[\sup_{r\in[t,T]}|X_{r}^{t,x}-X_{r}^{t,x'}|^p\bigg].
		\end{align*}
		Hence, by Lemma~\ref{th:SDE_standard_estim} together with Kolmogorov's continuity criterion, $ \bigl(\int_{t}^{s}g(r,X_r^{t,x})\,dW_r\bigr)_{x\in\R^d} $ admits a continuous modification $ \big(\overline{\int_{t}^{s}g(r,X_r^{t,x})\,dW_r}\big)_{x\in\R^d} $.
		
		\textbf{Step 1.}
		We first prove the claim in the case where $ g\in C^{1,2}([0,T]\times \R^d;\R^{d\times d'}) $ and there exists a constant $ L>0 $ such that
		\[
		| g(r,x) -  g(r,x')|
		\leq L|x-x'|,\quad
		|\nabla g(r,x) - \nabla g(r',x')|
		\leq L(|r-r'| + |x-x'|),
		\]
		in which case $ |\nabla g(r,x)|\leq \sqrt{d}L $ and \eqref{eq:inverse_flow_stochastic_integral_1} holds for every $ x\in \R^d $. Let $ t<s $ and let $ \Delta: t= r_0<r_1<\cdots < r_n=s $ be a partition of $ [t,s] $, and set
		\[
		I^{x,\Delta} :=\sum_{j=0}^{n-1} g(r_j, X_{r_j}^{t,x})(W_{r_{j+1}}- W_{r_{j}}).
		\]
		Then, by the continuity of $r\mapsto g(r,X_r^{t,x})$, for each $ x $ we have $ I^{x,\Delta}\to \int_{t}^{s}g(r,X_r^{t,x})\,dW_r $ in probability as $ |\Delta|\to 0 $. Moreover, the BDG inequality and the Lipschitz condition on $ g $ yield that for any $ p\geq 1 $,
		\[
		\E[|I^{x,\Delta}-I^{x',\Delta}|]
		\leq c_p(L^2T)^{p/2}\E\bigg[\sup_{r\in[t,T]}|X_{r}^{t,x}-X_{r}^{t,x'}|^p\bigg],
		\]
		where $ c_p $ is the same constant as in Step 0. Therefore, by Lemma~\ref{th:SDE_standard_estim} and Lemma~\ref{th:random_field_and_convergence_in_probability}, we obtain the convergence in probability
		\[
		\overline{\int_{t}^{s}g(r,X_r^{t,x})\,dW_r}  \bigg|_{x=\widehat{X}_s^{t,y}}
		= \lim_{|\Delta|\to 0 } \bigg( I^{x,\Delta}\Big|_{x=\widehat{X}_s^{t,y}} \bigg).
		\]
		We compute the right-hand side. Observe that
		\begin{align*}
			I^{x,\Delta}
			&= \sum_{j=0}^{n-1}g(r_{j+1}, X_{r_{j+1}}^{t,x})(W_{r_{j+1}}- W_{r_{j}})
			- \sum_{j=0}^{n-1}(g(r_{j+1}, X_{r_{j+1}}^{t,x})-g(r_j, X_{r_j}^{t,x}))(W_{r_{j+1}}- W_{r_j})\\
			&=: I_1^{x,\Delta} - I_2^{x,\Delta}.
		\end{align*}
		First, by the flow property $X_r^{t}(\widehat{X}_s^{t,y})=\widehat{X}_s^{r,y}$ and   Remark~\ref{rem:backward_Ito_integral},
		\[
		I_1^{x,\Delta}\Big|_{x=\widehat{X}_s^{t,y}}
		= \sum_{j=0}^{n-1}g(r_{j+1}, \widehat{X}_s^{r_{j+1},y})(W_{r_{j+1}}- W_{r_j})
		\to \int_t^s g(r,\widehat{X}_s^{r,y})\,\widehat{d}W_r
		\]
		in probability as $|\Delta|\to0$.
		Next, write the $ i $-th component of $ I_2^{x,\Delta} $ as $ I_{2,i}^{x,\Delta} $. Applying It\^o's formula to $(g_{ik}(r,X_r^{t,x}))_{r\in[t,s]}$ and using the property of quadratic covariation, we obtain
		\begin{align*}
			I_{2,i}^{x,\Delta}
			&= \sum_{k=1}^{d'}\sum_{j=0}^{n-1}(g_{ik}(r_{j+1}, X_{r_{j+1}}^{t,x})-g_{ik}(r_j, X_{r_j}^{t,x}))(W_{r_{j+1}}^k- W_{r_j}^k)\\
			&\to \sum_{k=1}^{d'} \Big\langle  g_{ik}(\cdot,X_\cdot^{t,x}),W^k \Big\rangle_s
			=\sum_{k=1}^{d'}\sum_{l=1}^{d}\int_t^s  \partial_{l} g_{ik}(r,X_r^{t,x})\sigma_{lk}(r,X_r^{t,x})\,dr
			= \int_{t}^{s}\widehat{g}_i(r,X_r^{t,x})\, dr
		\end{align*}
		in probability as $ |\Delta|\to 0 $. Therefore, by Lemma~\ref{th:estim_of_approximation_of_quadratic_varitaions_for_g(r,X)} and Lemma~\ref{th:random_field_and_convergence_in_probability} we get the convergence in probability
		\[
		\lim_{|\Delta|\to 0 } \bigg( I_{2}^{x,\Delta}\Big|_{x=\widehat{X}_s^{t,y}} \bigg)
		= \int_{t}^{s}\widehat{g}(r,X_r^{t,x})\, dr\bigg|_{x=\widehat{X}_s^{t,y}}
		= \int_{t}^{s}\widehat{g}(r,\widehat{X}_s^{r,y})\, dr,
		\]
		where we used again the flow property. Thus \eqref{eq:inverse_flow_stochastic_integral_2} holds in this case.
		
		\textbf{Step 2.}
		For the general case, let $ (\rho_n)_{n\geq 1} $ and $ (\rho_n^1)_{n\geq 1} $ be sequences of mollifiers on $ \R^d $ and on $ \R $, respectively. Define $g^n$ and $\widehat{g}^n$ by
		\begin{align*}
			&g^{n}(r,x) := \int_{\R}\int_{\R^d}g(r-u,x-y)\mathbf{1}_{[0,T]}(r-u)\rho_n(y)\rho^1_n(u)\,dy\,du,\\
			&\widehat{g}_i^{n}(r,x) := \sum_{j,k}\sigma_{jk}(r,x)\partial_jg^n_{ik}(r,x),\quad i=1,\dots,d .
		\end{align*}
		Then, it can be verified from \eqref{eq:inverse_flow_stochastic_integral_3} that
		\begin{gather}
			|g^{n}(r,x)- g^{n}(r,x')|
			\leq K_n(r)^{1/2}|x-x'|,\label{eq:inverse_flow_stochastic_integral_4}\\
			|\nabla g^n(r,x)-\nabla g^n(r',x')|
			\leq \|K_n\|_{\infty}^{1/2}\|\nabla\rho_n\|_{L^1} |x-x'|
			+ \sqrt{Td}\|K\|_{L^1}^{1/2}\|\rho^1_n\|_{\mathrm{Lip}}|r-r'| \notag
		\end{gather}
		where  $ K_n(r):= \int_{\R}K(r-u)\mathbf{1}_{[0,T]}(r-u)\rho_n^1(u)\,du $ and $ \|\rho^1_n\|_{\mathrm{Lip}} $ denotes the Lipschitz constant of $ \rho_n^1 $. Therefore, by Step~1 we have
		\begin{align}\label{eq:inverse_flow_stochastic_integral_5}
			\overline{\int_{t}^{s}g^n(r,X_r^{t,x})\,dW_r} \bigg|_{x=\widehat{X}_s^{t,y}}
			= \int_{t}^{s}g^n(r,\widehat{X}_s^{r,y})\, \widehat{d}W_r - \int_{t}^{s}\widehat{g}^n(r,\widehat{X}_s^{r,y})\, dr.
		\end{align}
		We pass to the limit in this equality. By the dominated convergence theorem,
		\[
		\partial_jg^{n}_{ik}(r,x) = \int_{\R}\int_{\R^d}\partial_jg_{ik}(r-u,x-y)\mathbf{1}_{[0,T]}(r-u)\rho_n(y)\rho^1_n(u)\,dy\,du.
		\]
		Hence, we can apply Lemma~\ref{th:approximation_molifier_1} (after redefining $\partial_j g_{ik}(r,\cdot)$ on a Lebesgue null set for each fixed $r\in[0,T]$, if necessary) and obtain that, for any $ R>0 $,
		\[
		\lim_{n\to \infty}\bigg(\int_0^T \sup_{|x|\leq R}|g^n(r,x)-g(r,x)|^2 \,dr + \int_0^T \sup_{|x|\leq R}|\widehat{g}^n(r,x)-\widehat{g}(r,x)| \,dr\bigg)
		=0.
		\]
		Since $r\mapsto \widehat{X}_s^{r,y}$ is continuous and hence bounded almost surely, we then have
		\[
		\lim_{n\to \infty}\bigg(\int_t^T |g^n(r,\widehat{X}_s^{r,y})-g(r,\widehat{X}_s^{r,y})|^2 \,dr + \int_t^T|\widehat{g}^n(r,\widehat{X}_s^{r,y})-\widehat{g}(r,\widehat{X}_s^{r,y})| \,dr\bigg)
		=0 \quad \text{a.s.}
		\]
		Consequently,
		for the right-hand side of \eqref{eq:inverse_flow_stochastic_integral_5},
		\[
		\int_{t}^{s}g^n(r,\widehat{X}_s^{r,y})\, \widehat{d}W_r - \int_{t}^{s}\widehat{g}^n(r,\widehat{X}_s^{r,y})\, dr
		\to
		\int_{t}^{s}g(r,\widehat{X}_s^{r,y})\, \widehat{d}W_r - \int_{t}^{s}\widehat{g}(r,\widehat{X}_s^{r,y})\, dr
		\]
		in probability as $n\to\infty$. Similarly, 
		\[
		\int_{t}^{s}g^n(r,X_r^{t,x})\,dW_r \to \int_{t}^{s}g(r,X_r^{t,x})\,dW_r
		\]
		in probability as well.
		Then, in view of \eqref{eq:inverse_flow_stochastic_integral_4} and $\|K_n\|_{L^1}\le \|K\|_{L^1}$, the estimate from Step~0 together with Lemma~\ref{th:random_field_and_convergence_in_probability} implies that, in probability,
		\[
		\lim_{n\to \infty}\bigg(\overline{\int_{t}^{s}g^n(r,X_r^{t,x})\,dW_r}\bigg|_{x=\widehat{X}_s^{t,y}}\bigg)
		=\overline{\int_{t}^{s}g(r,X_r^{t,x})\,dW_r}\bigg|_{x=\widehat{X}_s^{t,y}} .
		\]
		Combining this with the convergence on the right-hand side of \eqref{eq:inverse_flow_stochastic_integral_5}, we may pass to the limit in \eqref{eq:inverse_flow_stochastic_integral_5}, and the desired conclusion follows.
	\end{proof}
	
	As announced before Lemma~\ref{th:inverse_flow_stochastic_integral}, we now obtain the It\^o equation for the inverse flow with the backward It\^o integral. This result is analogous to the last assertion of \cite[Chapter~II, Theorem~6.1]{Kunita1984} (see Remark~\ref{rem:inverse_flow_SDE} below).
	
	\begin{prop}\label{th:inverse_flow_SDE}
		Let Assumption \ref{assum:globally_lipschitz_coefficients} hold, and assume all conditions in \ref{item:asummption_of_main_theorem_1} of Section~\ref{section:Main_theorem}, except that the local Lipschitz continuity of $x\mapsto \widehat{\sigma}(r,x)$  is replaced by continuity. Then, for each $ (s,y)\in[0,T]\times\R^d $, one has
		\begin{align}\label{eq:inverse_flow_SDE}
			\widehat{X}_s^{t,y}
			= y - \int_t^s \big(b(r,\widehat{X}_s^{r,y})-\widehat{\sigma}(r,\widehat{X}_s^{r,y})\big)\,dr
			- \int_t^s \sigma(r,\widehat{X}_s^{r,y})\,\widehat{d}W_r,\quad t\in[0,s].
		\end{align}
	\end{prop}
	
	\begin{proof}
		Fix $ t\le s $, and let $ \big(\overline{\int_{t}^{s}\sigma(r,X_r^{t,x})\,dW_r}\big)_{x\in\R^d} $ denote a continuous modification of $\bigl(\int_{t}^{s}\sigma(r,X_r^{t,x})\,dW_r\bigr)_{x\in\R^d}$ that is continuous in $x$ (see Lemma~\ref{th:inverse_flow_stochastic_integral}). 
		Taking continuous modifications in $ x $ on both sides of SDE \eqref{eq:SDE} and substituting $ x=\widehat{X}_s^{t,y} $ give
		\[
		y
		= \widehat{X}_s^{t,y}
		+ \int_t^s b(r,X_r^{t,x})\,dr\bigg|_{x=\widehat{X}_s^{t,y}}
		+ \overline{\int_t^s \sigma(r,X_r^{t,x})\,dW_r}\bigg|_{x=\widehat{X}_s^{t,y}}.
		\]
		Therefore, the conclusion follows by Lemma~\ref{th:inverse_flow_stochastic_integral} and the flow property.
	\end{proof}
	
	\begin{rem}\label{rem:inverse_flow_SDE}
		\cite[Chapter~II, Theorem~6.1]{Kunita1984} is stated under the assumption of \emph{local} Lipschitz continuity of the coefficients, together with time continuity and additional smoothness.
		Its last assertion states that, if the solution maps form a stochastic flow of homeomorphisms, then the inverse flow satisfies the Itô equation \eqref{eq:inverse_flow_SDE} with the backward Itô integral.
		
		In contrast, we impose \emph{global} Lipschitz continuity in $ x $ with an $ L^1 $-in-time modulus (hence the flow-of-homeomorphisms property is already established), and we relax both the time continuity and smoothness assumptions.
		Since spatial derivatives are available only almost everywhere, we work with the representative $ \widehat{\sigma} $ for the correction term (which is additionally assumed to be continuous in space), and the weaker time regularity is handled alongside by Lemma~\ref{th:approximation_molifier_1}.
	\end{rem}

	\section{Weak Differentiability of Solutions to SDEs}\label{section:weak_diffrentability_of_sol_of_SDE}
	This section discusses the weak differentiability of solutions to SDEs and derives estimates for their weak derivatives. 
	Together with the result in the previous section, these form one of the key tools in the proof of the main theorem.
	
	Throughout this section, we fix $ t\in[0,T] $ and let $ \mathscr P $ denote the progressive $ \sigma $-field on $ \Omega\times [t,T] $. The terms weak derivative and weakly differentiable are understood in the Sobolev (distributional) sense.

	We begin by discussing joint measurability of the weak derivative.
	
	\begin{lem}
		Let $ (X_s^x)_{s\in[t,T],\,x\in \R^d} $ be an $ \R^d $-valued $ \mathscr P\otimes \mathscr B(\R^d) $-measurable process, and suppose that for every $ s\in[t,T] $, the map $ x\mapsto X_s^x $ is weakly differentiable almost surely.
		Then, there exists an $ \R^{d\times d} $-valued $ \mathscr P\otimes \mathscr B(\R^d) $-measurable process $ (\widetilde{\nabla X}{}_s^x)_{s\in[t,T],\,x\in \R^d} $ such that
		\begin{align}\label{eq:jointly_mble_version_of_weak_derivative_process}
			\widetilde{\nabla X}{}_s^x 
			= \nabla X_s^x,\quad \text{a.e.}\;x,\;\;\text{a.s.}, \;\;\text{for all }s\in[t,T]
		\end{align}
		where the right-hand side denotes the weak derivative of $ x\mapsto X_s^x $.
	\end{lem}
	
	\begin{proof}
		Let $ \Gamma := \{(\omega,s)\in \Omega \times [t,T] \mid x\mapsto X_s^x(\omega) \text{ is weakly differentiable} \} $. Then, by Lemma~\ref{th:jointly_mble_vesion_of_weak_derivative}, there exists an $ \R^{d\times d} $-valued $ \mathscr P\otimes \mathscr B(\R^d) $-measurable process $ (\widetilde{\nabla X}{}_s^x)_{s\in[t,T],\,x\in \R^d} $ such that
		\[
		\widetilde{\nabla X}{}_s^x(\omega) = \nabla X_s^x(\omega),\quad \text{a.e.}\;x,\;\;\text{for all }(\omega,s)\in \Gamma.
		\]
		Since by assumption, for every $ s\in[t,T] $, the set $ \{\omega\in \Omega \mid (\omega,s)\in \Gamma \} $ has full $ \Prob $-measure, the conclusion follows.
	\end{proof}
	
	\begin{dfn}\label{rem:jointly_mble_version_of_weak_derivative_process}
		Let $ (X_s^x)_{s\in[t,T],\,x\in \R^d} $ be a process satisfying the assumptions in the above lemma.
		We refer to the $ \mathscr P\otimes \mathscr B(\R^d) $-measurable process $ (\widetilde{\nabla X}{}_s^x)_{s\in[t,T],\,x\in \R^d} $ satisfying~\eqref{eq:jointly_mble_version_of_weak_derivative_process} as the \emph{$ \mathscr P\otimes \mathscr B(\R^d) $-measurable weak derivative} of $ (X_s^x)_{s\in[t,T],\,x\in \R^d} $. 
		Similarly, for random fields, we define the \emph{$ \mathscr F\otimes \mathscr B(\R^d) $-measurable weak derivative} in the same manner.
	\end{dfn}
	
	We now present the weak differentiability of solutions to SDEs  under the globally Lipschitz condition (see Assumption~\ref{assum:globally_lipschitz_coefficients}).
	The proof is given in Appendix~\ref{section:proof_of_th:differentiability_of_sol_of_SDE}.
	
	\begin{lem}\label{th:differentiability_of_sol_of_SDE}
		Suppose that $ b $ and $ \sigma $ satisfy Assumption~\ref{assum:globally_lipschitz_coefficients}.
		For each $ x\in \R^d $, let $ (X_s^x)_{s\in[t,T]} $ be the solution to the SDE
		\begin{align}\label{eq:differentiability_of_sol_of_SDE0}
			X_s^x = x + \int_t^s b(r,X_r^x)\, dr + \int_t^s \sigma(r,X_r^x)\,dW_r,\quad s\in[t,T],
		\end{align}
		and assume that $ (X_s^x)_{s\in[t,T],\,x\in \R^d} $ is $ \mathscr P\otimes \mathscr B(\R^d) $-measurable. Then for every $ s\in[t,T] $, the map $ x\mapsto X_s^x $ is weakly differentiable almost surely.
		Moreover, for any $ p\geq 2 $, there exist constants $ C_p,C_p'>0 $ depending only on $ p,\|K\|_{L^1},d $ such that
		\[
		\E [|X_s^x|^p ]
		\leq C_p(1+|x|^p),\quad
		\E \big[|\widetilde{\nabla X}{}_s^x|^p \big]
		\leq C_p',\quad \text{a.e.}\;x,\;\;\text{for all }s\in[t,T].
		\]
		In particular, for any $ p\geq 1 $ and $ s\in[t,T] $, the map $ x\mapsto X_s^x $ belongs to $ W^{1,p}_{\mathrm{loc}}(\R^d;\R^d) $ almost surely.
	\end{lem}

	Next, we derive the linear SDE for the Jacobian matrix in the sense of weak derivatives, with coefficients expressed explicitly in terms of those of the original SDE.
	
	\begin{prop}\label{th:weak_derivative_and_SDE}
		Suppose that $ b $ and $ \sigma $ satisfy Assumption~\ref{assum:globally_lipschitz_coefficients}, and for each $ x\in \R^d $, let $ (X_s^x)_{s\in[t,T]} $ be the solution to the SDE~\eqref{eq:differentiability_of_sol_of_SDE0}. We further assume the following:
		\begin{enumerate}
			\item $ (X_s^x)_{s\in[t,T],\,x\in \R^d} $ is $ \mathscr P\otimes \mathscr B(\R^d) $-measurable.
			\item For every $ s\in [t,T] $, it holds almost surely that the preimage of any Lebesgue null set under the map $ x\mapsto X_s^x $ is also a Lebesgue null set.
		\end{enumerate}
		Let $ \widetilde{\nabla b}, \widetilde{\nabla \sigma}_{(k)}:[0,T]\times \R^d\to \R^{d\times d},\;k=1,\dots,d' $ be any Borel functions satisfying
		\begin{align}\label{eq:weak_derivative_and_SDE4}
			\widetilde{\nabla b}(r,x)= \nabla b(r,x),\quad
			\widetilde{\nabla \sigma}_{(k)}(r,x)= \nabla \sigma_{(k)}(r,x),\quad \text{a.e.}\;x,\;\,\text{a.e.}\;r,
		\end{align}
		where $ \sigma_{(k)} $ denotes the $k$-th column of $ \sigma $.
		Then, there exists a $ \mathscr P\otimes \mathscr B(\R^d) $-measurable process $ (J_s^x)_{s\in[t,T],\,x\in \R^d} $ such that:
		\begin{enumerate}[wide,topsep=3.1pt,itemsep=-2.4pt,labelindent=8pt]
			\item[(i)] For every fixed $ s\in[t,T] $, we have $ J_s^x = \widetilde{\nabla X}{}_s^x $ a.s., a.e.\:$x$.
			\item[(ii)] For almost every $ x\in \R^d $, the process $ (J_s^x)_{s\in[t,T]} $ satisfies the SDE
			\begin{align*}
				J_s^x
				= I + \int_t^s \widetilde{\nabla b}(r,X_r^x)J_r^x\, dr + \sum_{k=1}^{d'}\int_t^s \widetilde{\nabla \sigma}_{(k)}(r,X_r^x)J_r^x\,dW^k_r,\quad s\in[t,T],
			\end{align*}
			where $ I $ denotes the $d$-dimensional identity matrix.
			\item[(iii)] For almost every $ x\in \R^d $, it holds almost surely that for all $ s\in[t,T] $,
			\begin{align*}
				&\det J_s^x\\
				&=  \exp{\bigg(\int_t^s \!\bigg(\!\tr \widetilde{\nabla b}(r,X_r^x)-\frac{1}{2}\sum_{k=1}^{d'} \tr \big[(\widetilde{\nabla \sigma}_{(k)}(r,X_r^x))^2\big]\bigg) dr  + \sum_{k=1}^{d'}\int_t^s \tr \widetilde{\nabla \sigma}_{(k)}(r,X_r^x) \,dW^k_r \bigg)}.
			\end{align*}
		\end{enumerate}
	\end{prop}
	
	\begin{proof}
		By Lemma~\ref{th:differentiability_of_sol_of_SDE}, we have that for every $ s\in[t,T] $, the map $ x\mapsto X_s^x $ is weakly differentiable almost surely, and there exists a constant $ C>0 $ such that
		\[
		\E [|X_s^x|^2 ]
		\leq C(1+|x|^2),\quad
		\E [|\widetilde{\nabla X}{}_s^x|^2 ]
		\leq C,\quad \text{a.e.}\:x,\;\;\text{for all } s\in[t,T],
		\]
		where $ (\widetilde{\nabla X}{}_s^x)_{s\in[t,T],\,x\in \R^d} $ is the $ \mathscr P\otimes \mathscr B(\R^d) $-measurable weak derivative of $ (X_s^x)_{s\in[t,T],\,x\in \R^d} $.
		
		Then for any bounded set $ B\in \mathscr B(\R^d) $, we have
		\begin{align}\label{eq:weak_derivative_and_SDE3}
			&\int_{B}\E\bigg[\int_{t}^T (|b(r,X_r^x)|+|\sigma_{(k)}(r,X_r^x)|^2)\, dr\bigg]\,dx
			<\infty,\notag\\
			&\int_{B}\E\bigg[\int_{t}^T (|\widetilde{\nabla b}(r,X_r^x)\widetilde{\nabla X}{}_r^x |+|\widetilde{\nabla \sigma}_{(k)}(r,X_r^x)\widetilde{\nabla X}{}_r^x|^2)\, dr\bigg]\,dx
			<\infty.
		\end{align}
		The first claim follows from the moment estimate and the linear growth condition. For the second, note that by Assumption~\ref{assum:globally_lipschitz_coefficients} and~\eqref{eq:weak_derivative_and_SDE4},
		\[
		|\widetilde{\nabla b}(r,x)|\leq \sqrt{d}K(r),\quad
		|\widetilde{\nabla \sigma}_{(k)}(r,x)|\leq \sqrt{d}K(r)^{1/2},\quad\text{a.e.}\;x,\;\;\text{a.e.}\;r.
		\]
		In view of assumption (2), we then have
		\[
		|\widetilde{\nabla b}(r,X_r^x)|\leq \sqrt{d}K(r),\quad
		|\widetilde{\nabla \sigma}_{(k)}(r,X_r^x)|\leq \sqrt{d}K(r)^{1/2},\quad\text{a.e.}\;x,\;\;\text{a.s.},\;\;\text{a.e.}\;r
		\]
		and the required integrability follows from Fubini's theorem and the moment bound on $ \widetilde{\nabla X}{}_r^x $.
		
		For any process $ (Y_s^x)_{s\in[t,T]} $ defined for almost every $ x\in \R^d $, we use the notation $ (\roverline{Y_s^x})_{s\in[t,T],\,x\in \R^d} $ to denote any $ \mathscr P\otimes \mathscr B(\R^d) $-measurable process satisfying
		\[
		\roverline{Y_s^x} = Y_s^x,\quad s\in[t,T],\;\;\text{a.s.},\;\;\text{a.e.}\:x,
		\]
		if such a version exists.		
		Now define the $ \mathscr P\otimes \mathscr B(\R^d) $-measurable process $ (J_s^x)_{s\in[t,T],\,x\in \R^d} $ by
		\[
		J_s^x:= I + \roverline{\int_t^s \widetilde{\nabla b}(r,X_r^x)\widetilde{\nabla X}{}_r^x\, dr} + \sum_{k=1}^{d'}\roverline{\int_t^s \widetilde{\nabla \sigma}_{(k)}(r,X_r^x)\widetilde{\nabla X}{}_r^x\,dW^k_r}.
		\]
		(This is well-defined in view of~\eqref{eq:weak_derivative_and_SDE3}.) We claim that
		\begin{align}\label{eq:weak_derivative_and_SDE2}
			\widetilde{\nabla X}{}_s^x = J_s^x,\quad\text{a.e.}\;x,\;\;\text{a.s.},\;\;\text{for all } s\in[t,T].
		\end{align}
		Fix $ s\in[t,T] $ and $ \varphi \in C_c^\infty(\R^d;\R^d) $. By Fubini's theorem and the stochastic Fubini theorem,
		\begin{align*}
			&\int_{\R^d} X_s^x \divg\varphi(x)\,dx \\
			&=\int_{\R^d} \bigg(x + \int_t^s b(r,X_r^x)\, dr + \sum_{k=1}^{d'}\roverline{\int_t^s \sigma_{(k)}(r,X_r^x)\,dW^k_r}\, \bigg)\divg \varphi(x)\,dx \\
			&=  -\int_{\R^d}I\varphi(x)\,dx + \int_t^s \int_{\R^d}b(r,X_r^x)\divg \varphi(x)\,dx\, dr + \sum_{k=1}^{d'}\int_t^s \int_{\R^d}\sigma_{(k)}(r,X_r^x)\divg \varphi(x)\,dx\,dW^k_r
		\end{align*}
		holds almost surely. Moreover, by assumption (2) and Lemma~\ref{th:chain_rule_for_weak_derivative_Lipschitz}, we have
		\begin{align*}
			&\int_{\R^d}b(r,X_r^x)\divg \varphi(x)\,dx 
			=-\int_{\R^d}\widetilde{\nabla b}(r,X_r^x)\widetilde{\nabla X}{}_r^x \varphi(x)\,dx,\\
			&\int_{\R^d}\sigma_{(k)}(r,X_r^x)\divg \varphi(x)\,dx 
			=-\int_{\R^d}\widetilde{\nabla \sigma}_{(k)}(r,X_r^x)\widetilde{\nabla X}{}_r^x \varphi(x)\,dx,\quad \text{a.s.},\;\;\text{a.e.}\;r.
		\end{align*}
		Then, applying Fubini and stochastic Fubini again yields
		\begin{align*}
			&\begin{aligned}
				\int_{\R^d} X_s^x \divg\varphi(x)\,dx
				= -\int_{\R^d}I\varphi(x)\,dx &-\int_t^s \int_{\R^d}\widetilde{\nabla b}(r,X_r^x)\widetilde{\nabla X}{}_r^x \varphi(x)\,dx\,dr\\
				& - \sum_{k=1}^{d'}\int_t^s \int_{\R^d}\widetilde{\nabla \sigma}_{(k)}(r,X_r^x)\widetilde{\nabla X}{}_r^x \varphi(x)\,dx\,dW^k_r
			\end{aligned}\\
			&= -\int_{\R^d} \bigg(I + \roverline{\int_t^s \widetilde{\nabla b}(r,X_r^x)\widetilde{\nabla X}{}_r^x\, dr} + \sum_{k=1}^{d'}\roverline{\int_t^s \widetilde{\nabla \sigma}_{(k)}(r,X_r^x)\widetilde{\nabla X}{}_r^x\,dW^k_r}\, \bigg)\varphi(x)\,dx,\quad\text{a.s.}
		\end{align*}
		Since $ C^\infty(\R^d;\R^d) $ is separable, there exists a $\Prob$-null set outside of which the above equality holds for all $ \varphi\in C_c^\infty(\R^d;\R^d) $. This implies~\eqref{eq:weak_derivative_and_SDE2}.
		
		From~\eqref{eq:weak_derivative_and_SDE2}, we obtain (i) as well as
		\begin{align*}
			\widetilde{\nabla X}{}_s^x = J_s^x,\quad\text{a.e.}\;s,\;\;\text{a.s.},\;\;\text{a.e.}\;x.
		\end{align*}
		By combining the definition of $ (J_s^x)_{s\in[t,T],\,x\in \R^d} $ with the above identity, we deduce (ii).
		
		Finally, (iii) follows from the stochastic version of the Liouville formula (see, e.g., \cite[Chapter~3, Theorem~2.2]{Mao2008}\footnote{In \cite{Mao2008}, the Liouville formula is stated for linear matrix-valued SDEs with bounded and non-random coefficient. In our setting, the coefficients in (ii) are random, but the same proof applies under suitable integrability conditions.}
		 or \cite{Vrkov1978}).
	\end{proof}
	
	\begin{rem}
		One alternative approach to weak differentiability for Lipschitz SDEs---and to Jacobian representations as in Proposition~\ref{th:weak_derivative_and_SDE}~(ii)---relies on Dirichlet forms; see Bouleau--Hirsch~\cite{BouleauHirsch1988}.
		Their results are formulated on product spaces for SDEs with time-independent Lipschitz constants and in an $L^2$-framework. 
		
		In our setting, however, we allow time-dependent Lipschitz constants and require $L^p$-type moment estimates for the Jacobian, so the results of \cite{BouleauHirsch1988} do not apply directly.
		For this reason, and to keep our argument more self-contained, we have stated and proved Lemma~\ref{th:differentiability_of_sol_of_SDE} and Proposition~\ref{th:weak_derivative_and_SDE} above, using a more direct argument.
	\end{rem}
	
	Using the explicit determinant representation of the Jacobian matrix in Proposition~\ref{th:weak_derivative_and_SDE}~(iii), we derive the following bounds.
	
	\begin{lem}\label{th:estimate_of_determinant_of_sol_of_SDE}
		Under the same assumptions as in Proposition~\ref{th:weak_derivative_and_SDE}, suppose further that there exist a nonnegative function $ \widetilde{K} \in L^1([0,T]) $ and a constant $ N > 0 $ such that
		\begin{align}\label{eq:estimate_of_determinant_of_sol_of_SDE1}
			\begin{aligned}
				&\bigg|\tr \nabla b(r,x)-\frac{1}{2}\sum_{k=1}^{d'}\tr \big[(\nabla \sigma_{(k)}(r,x))^2\big]\bigg|\leq \widetilde{K}(r),\quad
				\sum_{k=1}^{d'}\big(\!\tr \nabla \sigma_{(k)}(r,x)\big)^2\leq \widetilde{K}(r),\\ 
				&\text{for a.e.}\;x\, \text{ with } |x|\leq N,\;\;\text{a.e.}\;r\in [0,T].
			\end{aligned}
		\end{align}
		Define the stopping time
		\[
		\tau_N^x := \inf\{ s \in [t,T] \mid |X_s^x| \geq N \} \wedge T.
		\]
		Then the following hold:
		\begin{enumerate}
			\item For any fixed $ s \in [t,T] $, we have $ \det \widetilde{\nabla X}{}_s^x \neq 0 $ almost surely for almost every $ x $.
			\item Let $ \alpha \in \R $ and set $ c_\alpha := \alpha^2/2 + |\alpha| $. Then
			\begin{align*}
				\E\big[\big|\!\det \widetilde{\nabla X}{}_s^x\big|^\alpha \mathbf{1}_{\{\tau_{N}^x\geq s\}}\big]
				\leq e^{c_\alpha \|\widetilde{K}\|_{L^1}},\quad \text{a.e.}\;x,\;\;\text{for all } s\in[t,T].
			\end{align*}
			In particular, if \eqref{eq:estimate_of_determinant_of_sol_of_SDE1} holds for all $ N > 0 $, where $ \widetilde{K}$ is independent of $N$, then
			\begin{align*}
				\E\big[\big|\!\det \widetilde{\nabla X}{}_s^x\big|^\alpha \big]
				\leq e^{c_\alpha \|\widetilde{K}\|_{L^1}},\quad \text{a.e.}\;x,\;\;\text{for all } s\in[t,T].
			\end{align*}
		\end{enumerate}
	\end{lem}
	
	\begin{proof}
		Let us choose Borel functions $\widetilde{\nabla b}, \widetilde{\nabla \sigma}_{(k)} : [0,T] \times \R^d \to \R^{d \times d}$ for $k = 1, \dots, d'$ satisfying~\eqref{eq:weak_derivative_and_SDE4}, such that for all $(r,x) \in [0,T] \times \R^d$ with $|x| \leq N$,
		\begin{align*}
			\bigg|\tr \widetilde{\nabla b}(r,x)-\frac{1}{2}\sum_{k=1}^{d'}\tr \big[(\widetilde{\nabla \sigma}_{(k)}(r,x))^2\big]\bigg|\leq \widetilde{K}(r),\quad
			\sum_{k=1}^{d'}\big(\!\tr \widetilde{\nabla \sigma}_{(k)}(r,x)\big)^2\leq \widetilde{K}(r).
		\end{align*}
		Then, by Proposition~\ref{th:weak_derivative_and_SDE}, there exists a $ \mathscr{P} \otimes \mathscr{B}(\R^d) $-measurable process $ (J_s^x)_{s \in [t,T],\, x \in \R^d} $ such that
		\begin{align}\label{eq:estimate_of_determinant_of_sol_of_SDE}
			\widetilde{\nabla X}{}_s^x = J_s^x,\quad\text{a.s.},\;\;\text{a.e.}\;x,\;\;\text{for all } s\in[t,T],
		\end{align}
		and such that for almost every $ x\in\R^d $, it holds almost surely that the identity
		\begin{align*}
			&\det J_s^x\\
			&=  \exp{\bigg(\int_t^s \bigg(\!\tr \widetilde{\nabla b}(r,X_r^x)-\frac{1}{2}\sum_{k=1}^{d'} \tr \big[(\widetilde{\nabla \sigma}_{(k)}(r,X_r^x))^2\big]\bigg) dr  + \sum_{k=1}^{d'}\int_t^s \tr \widetilde{\nabla \sigma}_{(k)}(r,X_r^x) \,dW^k_r \bigg)}
		\end{align*}
		holds for all $ s \in [t,T] $.
		
		Fix such $ x $. Then in particular, $ \det J_s^x > 0 $ for all $ s \in [t,T] $ almost surely, and
		\begin{align*}
			&|\det J_s^x|^\alpha\\
			&\!= \exp{\bigg(\!\int_t^s \bigg(\alpha\bigg(\!\tr \widetilde{\nabla b}(r,X_r^x)-\frac{1}{2}\sum_{k=1}^{d'} \tr [(\widetilde{\nabla \sigma}_{(k)}(r,X_r^x))^2]\bigg) + \frac{\alpha^2}{2}\sum_{k=1}^{d'}(\tr \widetilde{\nabla \sigma}_{(k)}(r,X_r^x))^2\bigg) \,dr \bigg)}\\
			&\quad \times \exp{\bigg( \alpha \sum_{k=1}^{d'}\int_t^s \tr \widetilde{\nabla \sigma}_{(k)}(r,X_r^x) \,dW^k_r -\frac{\alpha^2}{2}\sum_{k=1}^{d'}\int_t^s (\tr \widetilde{\nabla \sigma}_{(k)}(r,X_r^x))^2 \,dr \bigg)}\\
			&\!=: H^x_s M^x_s.
		\end{align*}
		Note that $ |X_r^x| \mathbf{1}_{\{r \leq \tau_N^x\}} \leq N $ for $ r > t $, so by the choice of $ \widetilde{\nabla b} $ and $ \widetilde{\nabla \sigma}_{(k)} $,
		\begin{align*}
			H^x_{s\wedge \tau_N^x}
			&\leq \exp \bigg(|\alpha| \int_t^{s}\widetilde{K}(r)\,dr +  \frac{\alpha^2}{2}\int_t^{s} \widetilde{K}(r) \,dr \bigg).
		\end{align*}
		Similarly, we have
		\begin{align*}
			\E\bigg[ \exp\bigg( \frac{\alpha^2}{2}\sum_{k=1}^{d'}\int_t^{s\wedge \tau_N^x} (\tr \widetilde{\nabla \sigma}_{(k)}(r,X_r^x))^2 \,dr \bigg) \bigg]
			\leq \exp \bigg(\frac{\alpha^2}{2}\int_t^{s} \widetilde{K}(r) \,dr \bigg) <\infty,
		\end{align*}
		and hence $ (M_{s \wedge \tau_N^x}^x)_{s \in [t,T]} $ is a true martingale. Consequently, $ \E[M_{s \wedge \tau_N^x}^x] = 1 $, and thus
		\begin{align*}
			\E[|\det J_s^x|^\alpha \mathbf{1}_{\{\tau_{N}^x\geq s\}}]\leq
			\E[|\det J_{s\wedge \tau_N^x}^x|^\alpha]
			&=  \E[H^x_{s\wedge \tau_N^x} M^x_{s\wedge \tau_N^x}]
			\leq \exp \Big(\Big(|\alpha|+\frac{\alpha^2}{2}\Big)\|\widetilde{K}\|_{L^1} \Big).
		\end{align*}
		
		Now, by \eqref{eq:estimate_of_determinant_of_sol_of_SDE}, assertion~(1) follows and, for any fixed $ s \in [t,T] $, we have
		\begin{align*}
			\E[|\det \widetilde{\nabla X}{}_s^x|^\alpha \mathbf{1}_{\{\tau_{N}^x\geq s\}}]
			&= \E[|\det J_s^x|^\alpha \mathbf{1}_{\{\tau_{N}^x\geq s\}}]
			\leq  \exp \Big(\Big(|\alpha|+\frac{\alpha^2}{2}\Big)\|\widetilde{K}\|_{L^1} \Big),\quad \text{a.e.}\;x.
		\end{align*}
		The final assertion follows from Fatou's lemma.
	\end{proof}

	\section{Proof of the main theorem}	\label{section:proof_of_main_theorem}
	We begin by stating several lemmas from Sobolev space theory and measure theory that are needed for the proof of the main theorem.
	
	\begin{lem}\label{th:regularity_of_cont_Sobolev_function}
		Let $ p > d $, and let $u:\R^d\to\R^d$ be a continuous map belonging to $W^{1,p}_{\mathrm{loc}}(\R^d;\R^d)$. Then:
		\begin{enumerate}
			\item The function $ u $ is differentiable almost everywhere, and its weak derivative coincides almost everywhere with the classical derivative.
			\item The function $ u $ maps every Lebesgue null set to a Lebesgue null set.
		\end{enumerate}
	\end{lem}
	
	\begin{proof}
		Statement (1) is Theorem~5 in Section~5.8 of \cite{Evans1998}.
		For (2), Theorem~4.2 in \cite{HenclKoskela2014} asserts that for any Lebesgue null set $ E\subset \R^d $ and any open ball $ B\subset \R^d $, the image $ u(E\cap B) $ is also a Lebesgue null set. Hence, $ u(E) $ itself is a Lebesgue null set.
	\end{proof}
	
	\begin{lem}[Change-of-variables formula {\cite[Theorem 8.21]{Leoni2017}}]\label{th:change_of_variables_formula}
		Let $ U\subset \R^d $ be open, and let $ \Psi:U\to \R^d $ be a continuous function.
		Suppose there exist Lebesgue measurable sets $ F,G\subset \R^d $ such that $ \Psi $ is differentiable at every point of $ F $ and injective on $ G $, and such that the sets $ U\setminus F $, $ \Psi(U\setminus F) $, and $ \Psi(U\setminus G) $ are all Lebesgue null sets.
		Then, for any Lebesgue measurable set $ E\subset U $ and any Lebesgue measurable function $ f:\Psi(E)\to[0,\infty] $, we have
		\[
		\int_{\Psi(E)}f(x)\,dx
		= \int_E f(\Psi(y))|\det \nabla \Psi(y)|\,dy.
		\]
	\end{lem}
	
	\begin{rem}
		Suppose that $ \Psi:\R^d\to \R^d $ is a homeomorphism such that $ \Psi\in W^{1,p}_{\mathrm{loc}}(\R^d;\R^d) $ for some $ p>d $.
		Then, taking into account Lemma~\ref{th:regularity_of_cont_Sobolev_function}, the map $ \Psi $ satisfies the assumptions of the above lemma, and the change-of-variables formula holds. (By Lemma~\ref{th:regularity_of_cont_Sobolev_function}, the right-hand side remains valid whether $ \nabla\Psi $ is interpreted as the weak derivative or the classical derivative.)
	\end{rem}
	
	In addition, we prepare the following lemma concerning an a priori estimate for solutions to SDEs.
	
	\begin{lem}\label{th:apriori_estimate_of_rho(X)}
		Let $ b: [0,T] \times \R^d\to \R^d $ and $ \sigma: [0,T] \times \R^d\to \R^{d\times d'} $ be Borel measurable. 
		Let $ \rho:\R^d\to (0,\infty) $ be a $ C^2 $ function such that there exist a nonnegative function $ \widetilde{K}\in L^1([0,T]) $ and a constant $ N>0 $ satisfying
		\begin{align}
			\begin{aligned}\label{eq:apriori_estimate_of_rho(X)}
				&|\langle b(r,x),\nabla\rho(x) \rangle|\leq \widetilde{K}(r)\rho(x),\quad
				|\sigma(r,x)||\nabla\rho(x)|\leq \widetilde{K}(r)^{1/2}\rho(x),\\
				&|\sigma(r,x)|^2|\nabla^2\rho(x)|_{\mathrm{op}}\leq \widetilde{K}(r)\rho(x),\quad |x|\leq N,\;\;r\in[0,T].
			\end{aligned}
		\end{align}
		Let $ (t,x_0)\in[0,T]\times \R^d $, and let $ (X_s)_{s\in[t,T]} $ be the solution to the SDE
		\[
		X_s = x_0 + \int_t^s b(r,X_r)\, dr + \int_t^s \sigma(r,X_r)\,dW_r,\quad s\in[t,T].
		\]
		Define
		\[
		\tau_N :=\inf \{s\in [t,T] \mid |X_s|\geq N \} \wedge T.
		\]
		Then, for any $ \alpha \in \R $, setting $ c_\alpha:=\alpha^2/2 + 2|\alpha| $, we have
		\begin{align*}
			\E\bigg[\bigg| \frac{\rho(X_{s\wedge\tau_N})}{\rho(x_0)} \bigg|^\alpha \bigg]
			\leq e^{c_\alpha\|\widetilde{K}\|_{L^1}},\quad s\in[t,T].
		\end{align*}
		In particular, if \eqref{eq:apriori_estimate_of_rho(X)} holds for all $ N>0 $, where $ \widetilde{K}$ is independent of $N$, then
		\begin{align*}
			\E\bigg[\bigg| \frac{\rho(X_s)}{\rho(x_0)} \bigg|^\alpha \bigg]
			\leq e^{c_\alpha\|\widetilde{K}\|_{L^1}},\quad s\in[t,T].
		\end{align*}
	\end{lem}
	
	\begin{proof}
		Let $ f(x):=\rho(x)^{\alpha} $. Then $\nabla f(x) =  \alpha\rho(x)^{\alpha-1}\nabla\rho(x)$ and
		\begin{align*}
			&\begin{aligned}
				|\nabla^2 f(x)|_{\mathrm{op}}
				&= \big|\alpha(\alpha-1)\rho(x)^{\alpha-2}\nabla\rho(x)\nabla\rho(x)^{\top} + \alpha\rho(x)^{\alpha-1}\nabla^2\rho(x)\big|_{\mathrm{op}}\\
				&\leq |\alpha|(|\alpha|+1)\rho(x)^{\alpha-2}|\nabla\rho(x)|^2 + |\alpha|\rho(x)^{\alpha-1}|\nabla^2\rho(x)|_{\mathrm{op}}.
			\end{aligned}
		\end{align*}
		Hence, by \eqref{eq:apriori_estimate_of_rho(X)}, for $ |x|\leq N $ we obtain
		\begin{align*}
			&|\langle b(r,x),\nabla f(x) \rangle|
			\leq|\alpha|\widetilde{K}(r)f(x),\quad
			|\sigma(r,x)||\nabla f(x)|
			\leq |\alpha|\widetilde{K}(r)^{1/2}f(x),\\
			&|\sigma(r,x)|^2|\nabla^2 f(x)|_{\mathrm{op}}
			\leq (\alpha^2 + 2|\alpha|)\widetilde{K}(r)f(x).
		\end{align*}
				
		By It\^o's formula,
		\begin{align}\label{eq:apriori_estimate_of_rho(X)2}
			\begin{aligned}
				f(X_{s\wedge \tau_N})
				&= f(x_0)+ \int_t^{s\wedge \tau_N} \langle\nabla f(X_r), b(r,X_r) \rangle\, dr
				+ \int_t^{s\wedge \tau_N} \langle\nabla f(X_r), \sigma(r,X_r)\,dW_r \rangle\\
				&\quad +\frac{1}{2}\int_t^{s\wedge \tau_N}\tr[\nabla^2f(X_r)(\sigma\sigma^\top)(r,X_r)]\,dr.
			\end{aligned}
		\end{align}
		Since $ |\tr [\nabla^2f(x)(\sigma\sigma^\top)(r,x)]|\leq |\sigma(r,x)|^2|\nabla^2f(x)|_{\mathrm{op}} $ and, on the event $ \{\tau_N>t\} $, one has $ |X_{r\wedge \tau_N}|\leq N $, it follows that
		\begin{align*}
			&\int_t^{s\wedge \tau_N} \langle\nabla f(X_r), b(r,X_r) \rangle\, dr
			+ \frac{1}{2}\int_t^{s\wedge \tau_N} \tr [\nabla^2f(X_r)\sigma\sigma^\top(r,X_r)]\,dr\\
			&\leq |\alpha|\int_t^{s\wedge \tau_N} \widetilde{K}(r)f(X_r)\,dr + \frac{1}{2}(\alpha^2 + 2|\alpha|)\int_t^{s\wedge \tau_N}\widetilde{K}(r)f(X_r)\,dr.
		\end{align*}
		Moreover,
		\begin{align*}
			\E\bigg[\int_t^{s\wedge \tau_N}  |\nabla f(X_r)|^2| \sigma(r,X_r)|^2\,dr   \bigg]
			\leq  \alpha^2\|\widetilde{K}\|_{L^1}\sup_{|x|\leq  N}|f(x)|^2 <\infty,
		\end{align*}
		so the process $ (\int_t^{s\wedge \tau_N}  \langle\nabla f(X_r), \sigma(r,X_r)\,dW_r\rangle)_{s\in[t,T]} $ is a true martingale. Therefore, noting that $ c_\alpha=\alpha^2/2 + 2|\alpha| $, from \eqref{eq:apriori_estimate_of_rho(X)2} we have
		\begin{align*}
			\E[f(X_{s\wedge \tau_N})]
			\leq f(x_0)+  c_\alpha \int_t^{s}   \widetilde{K}(r)\E[f(X_{r\wedge \tau_N})]\,dr. 
		\end{align*}
		Hence, dividing both sides by $ f(x_0) $ and applying Gronwall's lemma yields the first inequality.
		The second inequality follows from Fatou's lemma.
	\end{proof}
	
	\begin{rem}\label{rem:apriori_estimate_of_rho(X)}
		\begin{enumerate}
			\item If $ \rho_1 $ and $ \rho_2 $ satisfy \eqref{eq:apriori_estimate_of_rho(X)}, then the product $ \rho_1\rho_2 $ also satisfies \eqref{eq:apriori_estimate_of_rho(X)} with $ \widetilde{K} $ replaced by $ 4\widetilde{K} $.			
			\item Let $ \beta\in \R $, and define $ \rho(x):=(1+|x|^2)^{\beta} $. Then, under assumption~\ref{item:asummption_of_main_theorem_2}, the function $ \rho $ satisfies \eqref{eq:apriori_estimate_of_rho(X)} for any $ N>0 $ with $ \widetilde{K} := 16(\beta^2+ 2|\beta|)K $.
		\end{enumerate}
	\end{rem}	
	
	With the preliminary results in place, we now proceed to prove the main theorem.
	
	\begin{proof}[Proof of Theorem~\ref{th:main_theorem_locally_Lipschitz}]
		We derive \eqref{eq:norm_eq_result2} from \eqref{eq:norm_eq_result1} via Fubini's theorem, so it suffices to prove \eqref{eq:norm_eq_result1}.
		
		\textbf{Step 1.} We approximate the coefficients by suitable globally Lipschitz functions. For each $n=1,2,\dots$, define
		\[
		L_n(r):= \sup_{\substack{x\neq x'\\|x|,|x'|\leq n}}\frac{|\widehat{\sigma}(r,x)-\widehat{\sigma}(r,x')|}{|x-x'|},\quad r\in[0,T].
		\]
		By assumption~\ref{item:asummption_of_main_theorem_1}, $L_n$ is finite-valued and Borel measurable. Since $\mathrm{Leb}(\{L_n>\lambda\})\to 0$ as $\lambda\to\infty$, the Borel--Cantelli lemma yields a sequence $(\lambda_n)_{n\geq 1}\subset[0,\infty)$ such that
		\[
		\lim_{n\to\infty}\mathbf{1}_{\{L_n\leq \lambda_n\}}(r)=1,\quad\text{a.e.}\;r\in[0,T].
		\]
		Let $\chi \in C_c^\infty(\R^d)$ satisfy $0 \leq \chi \leq 1$, $\chi(x) = 1$ for $|x| \leq 1$, and $\chi(x) = 0$ for $|x| \geq 2$, and set $\chi_n(x) := \chi(x/n)$. We now define
		\begin{align*}
			&b^n(r,x):= b(r,x)\chi_n(x)^2\mathbf{1}_{\{L_n\leq \lambda_n\}}(r),\\
			&\sigma^n(r,x):= \sigma(r,x)\chi_n(x)\mathbf{1}_{\{L_n\leq \lambda_n\}}(r),\\
			&\widehat{\sigma}^n(r,x):= \widehat{\sigma}(r,x)\chi_n(x)^2\mathbf{1}_{\{L_n\leq \lambda_n\}}(r) + \sigma(r,x)\sigma(r,x)^\top\nabla\chi_n(x)\chi_n(x)\mathbf{1}_{\{L_n\leq \lambda_n\}}(r)
		\end{align*}
		and set $\widehat{b}^n := b^n - \widehat{\sigma}^n$.
		
		\textbf{Step 2.} We next verify that these coefficients satisfy the conditions needed for the subsequent arguments. 
		Define
		\begin{align*}
			&\gamma:= \sup_{x\in\R^d}(1+|x|)|\nabla\chi(x)| + \sup_{x\in\R^d}(1+|x|)^2|\nabla^2\chi(x)|,\\
			&\widetilde{K}_n(r):=4(1+\gamma+\gamma^2)(K(r) + \widetilde{K}(r)+ \widehat{K}_{2n}(r)).
		\end{align*}
		Then $\lim_{n\to \infty}\|\widetilde{K}_n\|_{L^1}/\log n =0$ by assumption~\ref{item:asummption_of_main_theorem_4}, and the following conditions \ref{item:main_theorem_locally_Lipschitz1}–\ref{item:main_theorem_locally_Lipschitz4} are satisfied for each $n$. (We confirm these in Appendix~\ref{section:appendix_to_proof_of_main_theorem}.)
		
		\begin{enumerate}[label=(\arabic*)]
			\item \label{item:main_theorem_locally_Lipschitz1}
			$ (b^n,\sigma^n) $ and $ (\widehat{b}^n,\sigma^n) $ satisfy Assumption~\ref{assum:globally_lipschitz_coefficients}. Moreover,
			\begin{align*}
				&|b^n(r,x)|\leq K(r)(1+|x|), \quad
				|\sigma^n(r,x)|\leq K(r)^{1/2}(1+|x|),\\
				&|\widehat{b}^n(r,x)| \leq \widetilde{K}_n(r)(1+|x|),\quad x\in\R^d,\;\; r\in[0,T].
			\end{align*}
			
			\item \label{item:main_theorem_locally_Lipschitz2}
			Let $B\subset \R^d$ be a Borel set and $\widetilde{K}_*\in L^1([0,T])$. We set $\widetilde{K}_* = \widetilde{K}$ when $B = \{x \in \R^d \mid|x| \leq n\}$, 
			and $\widetilde{K}_* = \widetilde{K}_n$ when $B = \R^d$. Then
			\begin{align}
				\begin{aligned}\label{eq:main_theorem_locally_Lipschitz02}
					&|\langle \widehat{b}^n(r,x),\nabla\rho(x) \rangle|\leq \widetilde{K}_*(r)\rho(x),\quad
					|\sigma^n(r,x)||\nabla\rho(x)|\leq \widetilde{K}_*(r)^{1/2}\rho(x),\\
					&|\sigma^n(r,x)|^2|\nabla^2\rho(x)|_{\mathrm{op}}\leq \widetilde{K}_*(r)\rho(x),\quad x\in B,\;\;r\in[0,T],
				\end{aligned}
			\end{align}
			and
			\begin{align}\label{eq:main_theorem_locally_Lipschitz03}
				\begin{aligned}
					&\bigg|-\tr \nabla \widehat{b}^n(r,x)-\frac{1}{2}\sum_{k=1}^{d'}\tr \big[(\nabla \sigma_{(k)}^n(r,x))^2\big]\bigg|\leq \widetilde{K}_*(r),\\
					&\sum_{k=1}^{d'}\big(\!\tr \nabla \sigma_{(k)}^n(r,x)\big)^2\leq \widetilde{K}_*(r),\quad\text{for a.e.}\;x\in B,\;\;\text{a.e.}\;r\in [0,T].
				\end{aligned}
			\end{align}
						
			\item \label{item:main_theorem_locally_Lipschitz4}
			For each $N=1,2,\dots$, there exists $K_N'\in L^1([0,T])$ independent of $n$ such that
			\begin{align*}
				&|b^n(r,x)-b^n(r,x')| \leq K_N'(r)|x-x'|,\\
				&|\sigma^n(r,x)-\sigma^n(r,x')| \leq K_N'(r)^{1/2}|x-x'|,\quad |x|,|x'|\leq N,\;\; r\in[0,T].
			\end{align*}
			Moreover, \eqref{eq:main_theorem_locally_Lipschitz02} also holds when $\widehat{b}^n$ is replaced by $b^n$, with $B = \R^d$ and $\widetilde{K}_* = \widetilde{K}$.
		\end{enumerate}
		
		\textbf{Step 3.}
		For each $t\in[0,T]$ and $x\in\R^d$, let $(X_s^{t,x,n})_{s\in[t,T]}$ denote the solution to the SDE
		\[
		X_s^{t,x,n} = x + \int_t^s b^n(r,X_r^{t,x,n})\,dr + \int_t^s \sigma^n(r,X_r^{t,x,n})\,dW_r,\quad s\in[t,T].
		\]
		We take a modification $(\widetilde{X}_s^{t,x,n})_{0\le t\le s\le T,\,x\in\R^d}$ 
		of $(X_s^{t,x,n})_{0\le t\le s\le T,\,x\in\R^d}$ such that, for almost every $\omega\in\Omega$, 
		the following properties hold:
		\begin{enumerate}
			  \item $(t,s,x)\mapsto \widetilde{X}_s^{t,x,n}(\omega)$ is continuous;
			\item for any $0\le t\le s\le T$, the map $x\mapsto \widetilde{X}_s^{t,x,n}(\omega)$ is a homeomorphism;
			\item for any $0\le t\le r\le s\le T$ and $x\in\R^d$, it holds that
			\[
			\widetilde{X}_s^{r,y,n}(\omega)\big|_{y=\widetilde{X}_r^{t,x,n}(\omega)}
			=\widetilde{X}_s^{t,x,n}(\omega).
			\]
			\end{enumerate}
		(See Section~\ref{section:inverse_flow}.)  
		This modification can be chosen so that properties (1)--(3) hold for all $\omega\in\Omega$, and we use the same notation $(X_s^{t,x,n})_{0\le t\le s\le T,\,x\in\R^d}$ hereafter.
		
		We show that there exist constants $C_n>0$ with $C_n\to0$ as $n\to\infty$ such that, for all $0\le t\le s\le T$ and $\varphi\in C_c(\R^d)$,
		\begin{align}\label{eq:main_theorem_locally_Lipschitz2}
			\begin{aligned}
				&C^{-1}\int_{\R^d} |\varphi(x)|\rho(x)\,dx 
				- C^{-2}C_n\int_{\R^d} |\varphi(x)|\rho(x)(1+|x|)\,dx\\
				&\leq \int_{\R^d} \E[|\varphi(X_s^{t,x,n})|]\rho(x)\,dx
				\leq C\int_{\R^d} |\varphi(x)|\rho(x)\,dx
				+ C_n\int_{\R^d} |\varphi(x)|\rho(x)(1+|x|)\,dx,
			\end{aligned}
		\end{align}
		where $C:=e^{5\|\widetilde{K}\|_{L^1}}$.
		
		From \eqref{eq:asummption_of_main_theorem_1} and the definitions of $\widehat{\sigma}^n$ and $\sigma^n$, we can check that the following relations hold for each $r\in[0,T]$:
		\[
		\widehat{\sigma}^n_i(r,x)
		= \sum_{\substack{1\le j\le d\\1\le k\le d'}}
		\sigma^n_{jk}(r,x)\partial_j\sigma^n_{ik}(r,x),
		\quad\text{a.e.}\;x,\;\; i=1,\dots,d.
		\]
		Hence, by Proposition~\ref{th:inverse_flow_SDE}, for fixed $s\in[0,T]$ and each $y\in\R^d$,
		\[
		\widehat{X}_s^{s-t,y,n}
		= y - \int_0^t \widehat{b}^n(s-r,\widehat{X}_s^{s-r,y,n})\,dr
		- \int_0^t \sigma^n(s-r,\widehat{X}_s^{s-r,y,n})\,d\widetilde{W}_r,\quad t\in[0,s],
		\]
		where $y \mapsto \widehat{X}_s^{s-r,y,n}$ denotes the inverse map of $y \mapsto X_s^{s-r,y,n}$ for each $r \in [0,s]$, and $\widetilde{W}_r:=W_s-W_{s-r}$.
		
		By Step~2~\ref{item:main_theorem_locally_Lipschitz1} and Lemma~\ref{th:differentiability_of_sol_of_SDE}, for any $p>d$ and each $t\in[0,s]$, the map $y\mapsto \widehat{X}_s^{s-t,y,n}$ belongs to $W^{1,p}_{\mathrm{loc}}(\R^d;\R^d)$ almost surely.
		Therefore, by Lemma~\ref{th:change_of_variables_formula} (the change-of-variables formula) and the remark following it, we have
		\begin{align}
			\int_{\R^d}\E[|\varphi(X_s^{s-t,x,n})|]\rho(x)\,dx
			&= \E\bigg[\int_{\R^d}|\varphi(y)|\rho(\widehat{X}_s^{s-t,y,n})
			|\det\nabla\widehat{X}_s^{s-t,y,n}|\,dy\bigg]\notag\\
			&= \int_{\R^d}|\varphi(y)|\rho(y)
			\E\bigg[\frac{\rho(Y_s^{t,y,n})}{\rho(y)}
			|\det\widetilde{\nabla Y}{}_s^{t,y,n}|\bigg]dy,
			\label{eq:main_theorem_locally_Lipschitz3}
		\end{align}
		where $Y_s^{t,y,n}:=\widehat{X}_s^{s-t,y,n}$, and 
		$(\widetilde{\nabla Y}{}_s^{t,y,n})_{t\in[0,s],\,y\in\R^d}$ denotes the measurable weak derivative of $(Y_s^{t,y,n})_{t\in[0,s],\,y\in\R^d}$ (see Section~\ref{section:weak_diffrentability_of_sol_of_SDE}).
				
		Define
		\[
		\tau_n^y:=\inf\{r\in[0,s]\mid |Y_s^{r,y,n}|\ge n\}\wedge s.
		\]
		Let $\alpha \in \R$ be arbitrary. (Although $\alpha$ will be fixed to $\pm 1$ in the final estimate, we leave it unspecified for now, as multiples such as $2\alpha$ and $3\alpha$ will appear in the intermediate bounds.) Then, interpreting $0^\alpha := \infty$ when $\alpha < 0$, we have 
		\begin{align*}
			&\E\bigg[\bigg(\frac{\rho(Y_s^{t,y,n})}{\rho(y)}
			|\det\widetilde{\nabla Y}{}_s^{t,y,n}|\bigg)^\alpha\bigg]\\
			&=\E\bigg[\bigg(\frac{\rho(Y_s^{t,y,n})}{\rho(y)}
			|\det\widetilde{\nabla Y}{}_s^{t,y,n}|\bigg)^\alpha
			\mathbf{1}_{\{\tau_n^y\ge t\}}\bigg]
			+\E\bigg[\bigg(\frac{\rho(Y_s^{t,y,n})}{\rho(y)}
			|\det\widetilde{\nabla Y}{}_s^{t,y,n}|\bigg)^\alpha
			\mathbf{1}_{\{\tau_n^y<t\}}\bigg]\\
			&=: (\mathrm{I})+(\mathrm{II}).
		\end{align*}
		By Step~2~\ref{item:main_theorem_locally_Lipschitz2} and Lemma~\ref{th:apriori_estimate_of_rho(X)}, setting $c_\alpha:=\alpha^2/2+2|\alpha|$ and noting $\|\widetilde{K}\|_{L^1([0,s])}\le \|\widetilde{K}\|_{L^1}$, we obtain
		\[
		\E\bigg[\bigg|\frac{\rho(Y_s^{t\wedge\tau_n^y,y,n})}{\rho(y)}\bigg|^\alpha\bigg]
		\le e^{c_\alpha\|\widetilde{K}\|_{L^1}},\quad
		\E\bigg[\bigg|\frac{\rho(Y_s^{t,y,n})}{\rho(y)}\bigg|^\alpha\bigg]
		\le e^{c_\alpha\|\widetilde{K}_n\|_{L^1}}.
		\]
		On the other hand, by Step~2~\ref{item:main_theorem_locally_Lipschitz1}, 
		Lemma~\ref{th:differentiability_of_sol_of_SDE}, and Lemma~\ref{th:regularity_of_cont_Sobolev_function}~(2),
		for each $t\in[0,s]$, the map $y\mapsto X_s^{s-t,y,n}$ sends null sets to null sets almost surely, so its inverse $y\mapsto Y_s^{t,y,n}$ pulls back null sets to null sets almost surely.
		Hence, by Step~2~\ref{item:main_theorem_locally_Lipschitz2} and Lemma~\ref{th:estimate_of_determinant_of_sol_of_SDE},
		setting $c_\alpha':=\alpha^2/2+|\alpha|$, we have
		\[
		\E[|\det\widetilde{\nabla Y}{}_s^{t,y,n}|^\alpha\mathbf{1}_{\{\tau_n^y\ge t\}}]
		\le e^{c_\alpha'\|\widetilde{K}\|_{L^1}},\quad
		\E[|\det\widetilde{\nabla Y}{}_s^{t,y,n}|^\alpha]
		\le e^{c_\alpha'\|\widetilde{K}_n\|_{L^1}},\quad \text{a.e.}\;y.
		\]
		
		Now take $\alpha=\pm1$. Since $c_{2}=6$ and $c_{2}'=4$,
		\begin{align*}
			(\mathrm{I})
			&\le \E\bigg[\bigg|
			\frac{\rho(Y_s^{t\wedge\tau_n^y,y,n})}{\rho(y)}\bigg|^{2\alpha}\bigg]^{1/2}
			\E[|\det\widetilde{\nabla Y}{}_s^{t,y,n}|^{2\alpha}
			\mathbf{1}_{\{\tau_n^y\ge t\}}]^{1/2}\\
			&\le \big(e^{c_{2}\|\widetilde{K}\|_{L^1}}
			e^{c_{2}'\|\widetilde{K}\|_{L^1}}\big)^{1/2}
			= e^{5\|\widetilde{K}\|_{L^1}}=C,\quad \text{a.e.}\;y.
		\end{align*}
		Next, we estimate term $(\mathrm{II})$. By Step~2~\ref{item:main_theorem_locally_Lipschitz1} and Lemma~\ref{th:SDE_apriori1}, there exists a constant $c''>0$ independent of $n$ such that
		\[
		\Prob(\tau_n^y<t)
		\le \Prob\bigg(\sup_{r\in[0,t]}|Y_s^{r,y,n}|\ge n\bigg)
		\le \frac{1}{n^3}\E\bigg[\sup_{r\in[0,s]}|Y_s^{r,y,n}|^3\bigg]
		\le \frac{1}{n^3}c''e^{c''\|\widetilde{K}_n\|_{L^1}}(1+|y|^3).
		\]
		Hence
		\[
		(\mathrm{II})
		\le \E\bigg[\bigg|\frac{\rho(Y_s^{t,y,n})}{\rho(y)}\bigg|^{3\alpha}\bigg]^{1/3}
		\E[|\det\widetilde{\nabla Y}{}_s^{t,y,n}|^{3\alpha}]^{1/3}
		\Prob(\tau_n^y<t)^{1/3}
		\le C_n(1+|y|),\quad \text{a.e.}\;y,
		\]
		where
		\[
		C_n:=n^{-1}\big(2c''e^{(c_3+c_3'+c'')\|\widetilde{K}_n\|_{L^1}}\big)^{1/3}.
		\]
		Since $\|\widetilde{K}_n\|_{L^1}/\log n \to 0$ as $n \to \infty$, we have $C_n \to 0$. Combining the bounds for (I) and (II), we obtain, for $\alpha = \pm 1$,
		\begin{align}\label{eq:main_theorem_locally_Lipschitz4}
			\E\bigg[\bigg(\frac{\rho(Y_s^{t,y,n})}{\rho(y)}
			|\det\widetilde{\nabla Y}{}_s^{t,y,n}|\bigg)^\alpha\bigg]
			\le C + C_n(1+|y|),\quad \text{a.e.}\;y.		
		\end{align}
		
		Observe that, for any positive random variable $Z$ with $\E[Z^{-1}]<\infty$, we have $\E[Z]\ge \E[Z^{-1}]^{-1}$, and for $a,\beta>0$, the inequality $(a+\beta)^{-1}\ge a^{-1}-a^{-2}\beta$ holds.
		Then, applying the case $\alpha = -1$ in \eqref{eq:main_theorem_locally_Lipschitz4}, we have
		\begin{align*}
			\E\bigg[\frac{\rho(Y_s^{t,y,n})}{\rho(y)}
			|\det\widetilde{\nabla Y}{}_s^{t,y,n}|\bigg]
			&\ge \E\bigg[\bigg(
			\frac{\rho(Y_s^{t,y,n})}{\rho(y)}
			|\det\widetilde{\nabla Y}{}_s^{t,y,n}|\bigg)^{-1}\bigg]^{-1}\\
			&\ge (C + C_n(1+|y|))^{-1}\\
			&\ge C^{-1} - C^{-2}C_n(1+|y|),\quad \text{a.e.}\;y.
		\end{align*}
		Finally, by substituting the above lower bound and \eqref{eq:main_theorem_locally_Lipschitz4} (with $\alpha = 1$) into \eqref{eq:main_theorem_locally_Lipschitz3} and replacing $t$ by $s - t$, we obtain \eqref{eq:main_theorem_locally_Lipschitz2}.
		
		\textbf{Step 4.}
		We now take the limit as $n\to\infty$ in \eqref{eq:main_theorem_locally_Lipschitz2}, where we continue to assume $\varphi\in C_c(\R^d)$.
		First, by Step~2~\ref{item:main_theorem_locally_Lipschitz1}, \ref{item:main_theorem_locally_Lipschitz4}, and the stability of SDEs (Lemma~\ref{th:SDE_stability}), we have $X_s^{t,x,n} \to X_s^{t,x}$ in probability for each $x\in\R^d$. Since $\varphi$ is bounded and continuous, it follows that
		\begin{align*}
			\E[| \varphi(X_s^{t,x,n})|]\to \E[ |\varphi(X_s^{t,x})| ]\quad(n\to \infty).
		\end{align*}
		To proceed, define $\rho_0(x):=(1+|x|^2)^{d}$.
		Then by Step~2~\ref{item:main_theorem_locally_Lipschitz4}, Lemma~\ref{th:apriori_estimate_of_rho(X)} and Remark~\ref{rem:apriori_estimate_of_rho(X)}, there exists a constant $C_0>0$ independent of $n$ such that
		\begin{align*}
			\E[|(\rho_0(X_s^{t,x,n})\rho(X_s^{t,x,n}))^{-1}|]
			\leq C_0(\rho_0(x)\rho(x))^{-1}.
		\end{align*}
		Hence,
		\begin{align*}
			\E[| \varphi(X_s^{t,x,n})|]\rho(x)
			&\leq  \|\varphi\rho_0\rho\|_\infty \E[|(\rho_0(X_s^{t,x,n})\rho(X_s^{t,x,n}))^{-1}|]\rho(x)\\
			&\leq \|\varphi\rho_0\rho\|_\infty C_0(\rho_0(x))^{-1}.
		\end{align*}
		Since the right-hand side is integrable in $x$ and independent of $n$, the dominated convergence theorem applies.
		Therefore, letting $n \to \infty$ in \eqref{eq:main_theorem_locally_Lipschitz2}, we obtain
		\begin{align}\label{eq:main_theorem_locally_Lipschitz01}
			C^{-1}\int_{\R^d} |\varphi(x)|\rho(x)\,dx \leq\int_{\R^d} \E[|\varphi(X^{t,x}_s)|]\rho(x)\,dx \leq C \int_{\R^d} |\varphi(x)|\rho(x)\,dx.
		\end{align}
		
		\textbf{Step 5.}
		Finally, we show that \eqref{eq:main_theorem_locally_Lipschitz01} also holds for general $\varphi$. Consider first the case where $\varphi$ is the indicator function of a bounded open set. 
		Since such functions can be approximated from below by a monotone increasing sequence of nonnegative continuous functions with compact support, \eqref{eq:main_theorem_locally_Lipschitz01} follows by the monotone convergence theorem.
		In particular, for any fixed $N>0$ and any open set $B\subset \R^d$, the inequality \eqref{eq:main_theorem_locally_Lipschitz01} holds for $\varphi=\mathbf{1}_{B\cap (-N,N)^d}$. Then, by the monotone class theorem, the same inequality extends to all $B\in\mathscr B(\R^d)$.
		Therefore, applying the monotone convergence theorem and standard approximation by simple functions, we conclude that \eqref{eq:main_theorem_locally_Lipschitz01} holds for any Borel function $\varphi$.
	\end{proof}
	
	We finally prove Corollary~\ref{th:cor_of_main_theorem} of the main theorem.
	
	\begin{proof}[Proof of Corollary~\ref{th:cor_of_main_theorem}]
		By \eqref{eq:cor_of_main_theorem02}, we have for any $r\in[0,T]$,
		\begin{align*}
			|\nabla b(r,x)|\le \sqrt{d}K(r),\quad
			|\nabla \sigma(r,x)|\le \sqrt{d}K(r)^{1/2},\quad
			|\nabla \widehat{\sigma}(r,x)|\le \sqrt{d}K(r),\quad \text{a.e.}\;x.
		\end{align*}
		Noting $|\tr A|\le \sqrt{d}|A|$ and $|\tr(A^2)|\le |A|^2$ for any matrix $A\in\R^{d\times d}$, we obtain, for every $r\in[0,T]$ and almost every $x\in\R^d$,
		\begin{align}\label{eq:cor_of_main_theorem2}
			&\bigg|\tr(\nabla b(r,x) - \nabla\widehat{\sigma}(r,x))
			+ \frac{1}{2}\sum_{k=1}^{d'} \tr\big[(\nabla\sigma_{(k)}(r,x))^2\big]\bigg|
			+ \sum_{k=1}^{d'}\big(\!\tr\nabla\sigma_{(k)}(r,x)\big)^2
			\le \Big(\frac{5}{2}d + d^2\Big)K(r).
		\end{align}		
		On the other hand, by \eqref{eq:asummption_of_main_theorem_1}, we have for any $y\in\R^d$,
		\begin{align*}
			|\langle \widehat{\sigma}(r,x),y\rangle|
			\le |\sigma(r,x)||\nabla\sigma(r,x)||y|
			\le \sqrt{d}K(r)^{1/2}|\sigma(r,x)||y|,\quad \text{a.e.}\;x.
		\end{align*}
		Together with the continuity of $x\mapsto \widehat{\sigma}(r,x)$ and $x\mapsto\sigma(r,x)$, this yields
		\begin{align}\label{eq:cor_of_main_theorem1}
			|\widehat{\sigma}(r,x)|\le \sqrt{d}K(r)^{1/2}|\sigma(r,x)|,\quad x\in\R^d.
		\end{align}
		In particular, $|\widehat{\sigma}(r,x)|\le \sqrt{d}K(r)(1+|x|)$, so that assumption~\ref{item:asummption_of_main_theorem_4} is satisfied.
		
		\medskip
		(1)
		Combining Remark~\ref{rem:apriori_estimate_of_rho(X)}~(2) with
		\eqref{eq:cor_of_main_theorem1}, and further taking into account \eqref{eq:cor_of_main_theorem2}, we see that
		\[
		\widetilde{K}(r)
		:= \Big(16(\beta^2 + 2|\beta|) + 4\sqrt{d}(\beta^2 + 2|\beta|)^{1/2}\Big)K(r)
		+ \Big(\frac{5}{2}d + d^2\Big)K(r)
		\]
		satisfies \eqref{eq:main_theorem_locally_Lipschitz'} and \eqref{eq:main_theorem_locally_Lipschitz}.
		Hence, Theorem~\ref{th:main_theorem_locally_Lipschitz} applies.
		
		\medskip
		(2)
		We first consider the case where $F\in C^2(\R^d)$ and all its partial derivatives are bounded.
		Since $\rho(x) = e^{F(x)}$ by assumption, we have
		\begin{align*}
			|\nabla\rho(x)| \le \|\nabla F\|_\infty\rho(x),\quad
			|\nabla^2\rho(x)|_{\mathrm{op}}
			\le (\|\nabla F\|_\infty^2 + \|\nabla^2F\|_\infty)\rho(x).
		\end{align*}
		Then, combining assumption \eqref{eq:cor_of_main_theorem01} with the inequality \eqref{eq:cor_of_main_theorem1}, and further taking into account \eqref{eq:cor_of_main_theorem2}, we may set
		\[
		\widetilde{K}(r)
		:= \big((1+\sqrt{d})\|\nabla F\|_\infty
		+ \|\nabla F\|_\infty^2
		+ \|\nabla^2 F\|_\infty\big)K(r)
		+ \Big(\frac{5}{2}d + d^2\Big)K(r),
		\]
		so that \eqref{eq:main_theorem_locally_Lipschitz'} and \eqref{eq:main_theorem_locally_Lipschitz} hold. Therefore,
		Theorem~\ref{th:main_theorem_locally_Lipschitz} is applicable in this case.
		
		We now turn to a general $F$ satisfying the assumption of the corollary. Let $\chi\in C_c^\infty(\R^d)$ satisfy $0\le\chi\le1$, $\chi(x)=1$ for $|x|\le2$, and $\chi(x)=0$ for $|x|\ge3$,
		and define $\widetilde{F}(x):=F(x)(1-\chi(x/R_0))$. Then $\widetilde{F}$ is a $C^2$ function with bounded derivatives on $\R^d$. 
		Therefore, applying the preceding argument to the weight $e^{\widetilde{F}}$,
		we obtain constants $\widetilde{c},\widetilde{C}>0$ such that the estimates \eqref{eq:norm_eq_result1} and \eqref{eq:norm_eq_result2} holds with $\rho$ replaced by $e^{\widetilde{F}}$
		and with $c,C$ replaced by $\widetilde{c},\widetilde{C}$.
		Let $c_0 := \sup_{|x|\le3R_0}|F(x)|$. Then $ e^{-c_0}e^{F(x)}\le e^{\widetilde{F}(x)}\le e^{c_0}e^{F(x)} $ for all $ x\in \R^d $,
		so \eqref{eq:norm_eq_result1} and \eqref{eq:norm_eq_result2} for $\rho=e^{F}$ follows by taking $c:=\widetilde{c}e^{-2c_0}$ and $C:=\widetilde{C}e^{2c_0}$.
	\end{proof}
	
	\begin{rem}\label{rem:generalization_of_cor_2}
		As mentioned in Section~\ref{section:Main_theorem}, Corollary~\ref{th:cor_of_main_theorem} can be extended to the generalized equivalence of norm principle in \cite{ZhangZhao2007}, in which the functions $\varphi$ and $\psi$ are random.
		Indeed, under the assumptions of Corollary~\ref{th:cor_of_main_theorem}, the results of Sections~\ref{section:inverse_flow_main} and~\ref{section:weak_diffrentability_of_sol_of_SDE} apply directly,
		so that the approximation procedure used in the proof of Theorem~\ref{th:main_theorem_locally_Lipschitz} is no longer needed (see also the proof of Corollary~\ref{th:cor_of_main_theorem}). One can then follow the same argument as in \cite[Lemma~2.6]{ZhangZhao2007}.
	\end{rem}
	
	\section{Application}\label{section:application}
	As an application of the norm equivalence result (Theorem~\ref{th:main_theorem_locally_Lipschitz}), we derive integrability and non-integrability properties of solutions to certain PDEs via probabilistic methods.
	
	\begin{ex}
		Let $h:[0,T]\times\R^d\to\R^m$ and $f:[0,T]\times\R^d\times\R^m\times\R^{m\times d'}\to\R^m$ be Borel measurable functions.
		Under the assumptions of Theorem~\ref{th:main_theorem_locally_Lipschitz}, suppose that $u\in C^{1,2}([0,T]\times\R^d;\R^m)$ is a classical solution to the following parabolic PDE:
		\begin{align*}
			\left\{\begin{aligned}
				&\partial_t u(t,x) + \mathscr L u(t,x) + f(t,x,u(t,x),(\nabla u \sigma)(t,x)) = 0, \quad t \in (0,T),\\
				&u(T,x) = h(x),
			\end{aligned}
			\right.
		\end{align*}
		where $\mathscr L u := \frac{1}{2}\sum_{i,j}(\sigma\sigma^\top)_{ij}\partial_{i}\partial_{j}u + \sum_i b_i\partial_i u$, with $b$ and $\sigma$ as in Theorem~\ref{th:main_theorem_locally_Lipschitz}.
		Assume further that there exist a constant $L > 0$ and a Borel function $f_0:[0,T]\times\R^d\to\R$ such that
		\begin{align*}
			\langle y, f(t,x,y,z)\rangle
			\leq |f_0(t,x)||y| + L|y|^2 + L|z||y|,\quad (t,x,y,z)\in [0,T]\times\R^d\times\R^m\times\R^{m\times d'}.
		\end{align*}
		Suppose that $h(x)$ and $f_0(t,x)$ are of polynomial growth in $x$, uniformly in $t\in[0,T]$. If $x\mapsto u(t,x)$ also has polynomial growth uniformly in $t\in[0,T]$, then the following holds:		
		\begin{enumerate}
			\item For every $p>1$, there exists a constant $C_p>0$, depending only on $p$, $L$, $T$, and $\|\widetilde{K}\|_{L^1}$, such that, for every $t\in[0,T]$,
			\begin{align*}
				&\int_{\R^d}|u(t,x)|^p\rho(x)\,dx + \int_t^T\int_{\R^d}|(\nabla u\sigma)(s,x)|^2\rho(x)\,dx\,ds\,\mathbf{1}_{\{p=2\}}(p)\\
				&\leq C_p\int_{\R^d}|h(x)|^p\rho(x)\,dx + C_p\bigg(\int_t^T\bigg(\int_{\R^d}|f_0(s,x)|^p\rho(x)\,dx \bigg)^{1/p}\,ds\bigg)^p.
			\end{align*}			
			\item Suppose that $f(t,x,y,z)$ does not depend on $y$ and $z$ (so we may write $f(t,x)$), that $m=1$, and that $h$ and $f$ are nonnegative. Then $u$ is also nonnegative, and there exists a constant $C > 0$ such that, for every $t\in[0,T]$,
			\begin{align*}
				\int_{\R^d}u(t,x)\rho(x)\,dx
				\geq C\bigg( \int_{\R^d}h(x)\rho(x)\,dx  + \int_t^T\int_{\R^d}f(s,x)\rho(x)\,dx \,ds\bigg).
			\end{align*}
			In particular, if the right-hand side is infinite for some $t\in[0,T]$, then $\int_{\R^d}u(t,x)\rho(x)\,dx = \infty$ for such $t$.
		\end{enumerate}
	\end{ex}
	
	\begin{proof}
		Set $Y_s^{t,x} := u(s,X_s^{t,x})$, $Z_s^{t,x} := (\nabla u \sigma)(s,X_s^{t,x})$. Applying Itô's formula yields
		\begin{align}\label{eq:application_of_main_theorem2}
			Y_s^{t,x}
			= h(X_T^{t,x}) + \int_s^T f(r,X_r^{t,x},Y_r^{t,x},Z_r^{t,x})\,dr - \int_s^T Z_r^{t,x}\,dW_r,\quad s\in[t,T].
		\end{align}
		By the polynomial growth assumption on $u$ and Lemma~\ref{th:SDE_apriori1}, the standard a priori estimate for BSDEs (see \cite{Briand_etal.2003}) applies, and for every $p>1$ there exists a constant $c_{p,L,T}>0$ such that
		\begin{align}\label{eq:application_of_main_theorem}
			\E\bigg[\sup_{s\in[t,T]}|Y_s^{t,x}|^p + \bigg(\int_t^T|Z_r^{t,x}|^2\,dr \bigg)^{p/2} \bigg]
			\leq c_{p,L,T}\E\bigg[|h(X_T^{t,x})|^p + \bigg(\int_t^T|f_0(r,X_r^{t,x})|\,dr \bigg)^p \bigg].
		\end{align}
		
		(1)  
		By Theorem~\ref{th:main_theorem_locally_Lipschitz},
		\begin{align*}
			&\int_{\R^d}|u(t,x)|^p\rho(x)\,dx + \int_t^T\int_{\R^d}|(\nabla u\sigma)(s,x)|^2\rho(x)\,dx\,ds\,\mathbf{1}_{\{p=2\}}(p)\\
			&\leq e^{5\|\widetilde{K}\|_{L^1}}\bigg(\int_{\R^d}\E[|u(t,X_t^{t,x})|^p]\rho(x)\,dx + \int_t^T\int_{\R^d}\E[|(\nabla u\sigma)(s,X_s^{t,x})|^2]\rho(x)\,dx\,ds\,\mathbf{1}_{\{p=2\}}(p)  \bigg)\\
			&= e^{5\|\widetilde{K}\|_{L^1}}\int_{\R^d}\E\bigg[|Y_t^{t,x}|^p + \bigg(\int_t^T|Z_r^{t,x}|^2\,dr \bigg)^{p/2}\mathbf{1}_{\{p=2\}}(p)  \bigg] \rho(x)\,dx.
		\end{align*}
		By \eqref{eq:application_of_main_theorem} and Minkowski's integral inequality, this is bounded above by
		\begin{align*}
			&c_{p,L,T}e^{5\|\widetilde{K}\|_{L^1}}\bigg(\int_{\R^d}\E[|h(X_T^{t,x})|^p]\rho(x)\,dx  + \int_{\R^d}\E\bigg[ \bigg(\int_t^T|f_0(r,X_r^{t,x})|\,dr \bigg)^p \bigg]\rho(x)\,dx\bigg)\\
			&\leq c_{p,L,T}e^{5\|\widetilde{K}\|_{L^1}}\bigg(\int_{\R^d}\E[|h(X_T^{t,x})|^p]\rho(x)\,dx  +  \bigg(\int_t^T\bigg(\int_{\R^d}\E[|f_0(r,X_r^{t,x})|^p]\rho(x)\,dx\bigg)^{1/p} \,dr \bigg)^p\bigg).
		\end{align*}
		Therefore, applying Theorem~\ref{th:main_theorem_locally_Lipschitz} again, we obtain
		\begin{align*}
			&\int_{\R^d}|u(t,x)|^p\rho(x)\,dx + \int_t^T\int_{\R^d}|(\nabla u\sigma)(s,x)|^2\rho(x)\,dx\,ds\,\mathbf{1}_{\{p=2\}}(p)\\
			&\leq c_{p,L,T}e^{10\|\widetilde{K}\|_{L^1}}\bigg(\!\int_{\R^d}|h(x)|^p\rho(x)\,dx  +  \bigg(\int_t^T\bigg(\!\int_{\R^d}|f_0(r,x)|^p\rho(x)\,dx\bigg)^{1/p} \,dr \bigg)^p\bigg),
		\end{align*}
		as desired.
		
		(2)  
		By the polynomial growth assumptions on $h$ and $f_0$ and Lemma~\ref{th:SDE_apriori1}, the right-hand side of \eqref{eq:application_of_main_theorem} is finite.
		Hence, $(\int_t^s Z_r^{t,x}\,dW_r)_{s\in[t,T]}$ is a true martingale. Taking the expectation of both sides of \eqref{eq:application_of_main_theorem2} at $s=t$ yields
		\begin{align*}
			u(t,x)
			= \E[Y_t^{t,x}]
			= \E\bigg[h(X_T^{t,x}) + \int_t^T f(r,X_r^{t,x})\,dr \bigg].
		\end{align*}
		In particular, $u$ is nonnegative.
		Multiplying both sides of this equality by $\rho(x)$, integrating over $\R^d$, and applying
		Theorem~\ref{th:main_theorem_locally_Lipschitz} yields the desired estimate. 
	\end{proof}
	
	\appendix
	\renewcommand{\thesection}{\Alph{section}} 

	\section{A priori estimates and stability for SDEs}
	In this section, fix $ 0\leq t_0\leq T $.  
	Let $ \xi:\Omega\to\R^d $ be an $ \mathscr F_{t_0} $-measurable random variable,
	and let $ b:\Omega\times[t_0,T]\times\R^d\to\R^d $ and $ \sigma:\Omega\times[t_0,T]\times\R^d\to\R^{d\times d'} $  
	be $(\mathscr F_t)$-progressively measurable.
	We begin with the following basic a priori estimate.
	
	\begin{lem}\label{th:SDE_apriori1}
		Suppose there exist a nonnegative function $ K\in L^1([0,T]) $ and nonnegative progressively measurable processes  
		$ (b_0(r))_{r\in[t_0,T]} $ and $ (\sigma_0(r))_{r\in[t_0,T]} $ such that
		\begin{align*}
			|b(r,x)|\leq b_0(r)+K(r)|x|,\quad
			|\sigma(r,x)|\leq \sigma_0(r)+K(r)^{1/2}|x|,\quad x\in\R^d,\;\;r\in[t_0,T].
		\end{align*}
		Let $ (X_s)_{s\in[t_0,T]} $ be the solution to the SDE
		\begin{equation}\label{eq:SDE_apriori01}
			X_s = \xi + \int_{t_0}^s b(r,X_r)\,dr + \int_{t_0}^s \sigma(r,X_r)\,dW_r,\quad s\in[t_0,T].
		\end{equation}
		Then for every $ p\geq 2 $, there exists a constant $ C_p>0 $, depending only on $p$, such that
		\begin{align}\label{eq:SDE_apriori1}
			\E\bigg[\sup_{s\in[t_0,T]}|X_s|^p\bigg]
			\leq C_pe^{C_p\|K\|_{L^1}}
			\E\bigg[|\xi|^p+\bigg(\int_{t_0}^T|b_0(s)|\,ds\bigg)^p
			+\bigg(\int_{t_0}^T|\sigma_0(s)|^2\,ds\bigg)^{p/2}\bigg].
		\end{align}
		In particular, if $ b_0(r)\leq K(r) $ and $ \sigma_0(r)\leq K(r)^{1/2} $, then there exists another constant $ C_p'>0 $, depending only on $p$, such that
		\begin{align}\label{eq:SDE_apriori2}
			\E\bigg[\sup_{s\in[t_0,T]}|X_s|^p\bigg]
			\leq C_p'e^{C_p'\|K\|_{L^1}}\big(1+\E[|\xi|^p]\big).
		\end{align}
	\end{lem}
	
	\begin{proof}
		If $ b_0(r)\leq K(r) $ and $ \sigma_0(r)\leq K(r)^{1/2} $, then \eqref{eq:SDE_apriori2} follows from \eqref{eq:SDE_apriori1} by adjusting the constant appropriately.  
		To prove \eqref{eq:SDE_apriori1}, it suffices to consider the case where the right-hand side is finite.  
		In this case, \eqref{eq:SDE_apriori1} follows from the proof of \cite[Proposition~3.28]{PardouxRascanu2014}.  
		(Although that proposition assumes local Lipschitz continuity of the coefficients, only the linear growth condition is used for the estimate of the solution.)
	\end{proof}
	
	As a corollary, we also obtain the following a priori estimates for the difference of two solutions.
	
	\begin{lem}\label{th:SDE_apriori_diff}
		Let $ (\xi^1,b^1,\sigma^1):=(\xi,b,\sigma) $, and let another triple  
		$ (\xi^2,b^2,\sigma^2) $ satisfy the same assumptions as in the beginning of this section.  
		Suppose that there exists a nonnegative function $ K\in L^1([0,T]) $ such that, for all $ r\in[t_0,T] $ and $ x,x'\in\R^d $,
		\begin{align*}
			|b^2(r,x)-b^2(r,x')|\leq K(r)|x-x'|,\quad
			|\sigma^2(r,x)-\sigma^2(r,x')|\leq K(r)^{1/2}|x-x'|.
		\end{align*}
		If $ (X_s^i)_{s\in[t_0,T]} $, $ i=1,2 $, are the solutions to the SDEs
		\begin{align}\label{eq:SDE_apriori_diff}
			X_s^i = \xi^i + \int_{t_0}^s b^i(r,X_r^i)\,dr + \int_{t_0}^s \sigma^i(r,X_r^i)\,dW_r,\quad s\in[t_0,T],
		\end{align}
		then for every $ p\geq 2 $, there exists a constant $ C_p>0 $, depending only on $p$, such that
		\begin{align*}
			&\E\bigg[\sup_{s\in[t_0,T]}|X_s^1-X_s^2|^p\bigg]\\
			&\leq C_pe^{C_p\|K\|_{L^1}}
			\E\bigg[|\xi^1-\xi^2|^p
			+\bigg(\int_{t_0}^T|b^1-b^2|(r,X_r^1)\,dr\bigg)^p
			+\bigg(\int_{t_0}^T|\sigma^1-\sigma^2|^2(r,X_r^1)\,dr\bigg)^{p/2}\bigg].
		\end{align*}
	\end{lem}
	
	\begin{proof}
		It suffices to apply Lemma~\ref{th:SDE_apriori1} to the SDE satisfied by the difference $ (X_s^1-X_s^2)_{s\in[t_0,T]} $:
		\begin{align*}
			X_s^1 - X_s^2
			= \xi^1-\xi^2 + \int_{t_0}^s \widetilde{b}(r,X_r^1-X_r^2)\,dr
			+ \int_{t_0}^s \widetilde{\sigma}(r,X_r^1-X_r^2)\,dW_r,
		\end{align*}
		where
		\begin{align*}
			&\widetilde{b}(r,x):= b^1(r,X_r^1)-b^2(r,X_r^1)
			+ b^2(r,x+X_r^2)-b^2(r,X_r^2),\\
			&\widetilde{\sigma}(r,x):= \sigma^1(r,X_r^1)-\sigma^2(r,X_r^1)
			+ \sigma^2(r,x+X_r^2)-\sigma^2(r,X_r^2).\qedhere
		\end{align*}
	\end{proof}
	
	\begin{lem}\label{th:SDE_apriori_diff_loc}
		Let $ (\xi^1,b^1,\sigma^1):=(\xi,b,\sigma) $, and let another triple  
		$ (\xi^2,b^2,\sigma^2) $ satisfy the same assumptions as in the beginning of this section.  
		Assume that for each $ N=1,2,\dots $, there exists a nonnegative function $ K_N\in L^1([0,T]) $ such that
		\begin{align*}
			&|b^2(r,x)-b^2(r,x')|\leq K_N(r)|x-x'|,\\
			&|\sigma^2(r,x)-\sigma^2(r,x')|\leq K_N(r)^{1/2}|x-x'|,
			\quad |x|,|x'|\leq N,\;\; r\in[t_0,T].
		\end{align*}
		Let $ (X_s^i)_{s\in[t_0,T]} $, $ i=1,2 $, be the solutions to \eqref{eq:SDE_apriori_diff}, and define
		\[
		\tau_N := \inf\{s\in[t_0,T]\mid |X_s^1|\geq N\text{ or }|X_s^2|\geq N\}\wedge T.
		\]
		Then for every $ p\geq 2 $, there exists a constant $ C_p>0 $, depending only on $p$, such that
		\begin{align*}
			&\E\bigg[\sup_{s\in[t_0,T]}|X^1_{s\wedge\tau_N}-X^2_{s\wedge\tau_N}|^p\bigg]\\
			&\leq C_pe^{C_p\|K_N\|_{L^1}}
			\E\bigg[|\xi^1-\xi^2|^p
			+\bigg(\int_{t_0}^{\tau_N}|b^1-b^2|(r,X_r^1)\,dr\bigg)^p
			+\bigg(\int_{t_0}^{\tau_N}|\sigma^1-\sigma^2|^2(r,X_r^1)\,dr\bigg)^{p/2}\bigg].
		\end{align*}
	\end{lem}
	
	\begin{proof}
		For each $ N $, define $ \theta_N:\R^d\to\R^d $ by $ \theta_N(x):=(Nx)/(|x|\vee N) $ and set
		\begin{align*}
			b_N^i(r,x):=b^i(r,\theta_N(x))\mathbf{1}_{\{r\leq\tau_N\}},\quad
			\sigma_N^i(r,x):=\sigma^i(r,\theta_N(x))\mathbf{1}_{\{r\leq\tau_N\}}.
		\end{align*}
		Note that on the event $\{\tau_N>t_0\}$, one has $|X_{r\wedge\tau_N}^i|\leq N$.  
		Hence
		\begin{align*}
			X^i_{s\wedge\tau_N}
			= \xi^i + \int_{t_0}^s b_N^i(r,X^i_{r\wedge\tau_N})\,dr
			+ \int_{t_0}^s \sigma_N^i(r,X^i_{r\wedge\tau_N})\,dW_r.
		\end{align*}
		Therefore, the claim follows by applying Lemma~\ref{th:SDE_apriori_diff} to this SDE.
	\end{proof}
	
	Finally, we present the stability property for SDEs with locally Lipschitz coefficients.
	
	\begin{lem}[Stability]\label{th:SDE_stability}
		Let $(b,\sigma)$ and $(b^n,\sigma^n)$ for $n=1,2,\dots$ satisfy the assumptions \ref{item:asummption_of_main_theorem_2} and \ref{item:asummption_of_main_theorem_3} in Section~\ref{section:Main_theorem},  
		with the same functions $K,K_N\in L^1([0,T])$ for all $N\geq1$.  
		Let $x_0,x_n\in\R^d$ for $n=1,2,\dots$, and suppose $x_n\to x_0$ as $n\to\infty$.  
		Assume further that
		\[
		\lim_{n\to \infty}b^n(r,x)= b(r,x),\quad
		\lim_{n\to \infty}\sigma^n(r,x)= \sigma(r,x),\quad x\in\R^d,\;\;\text{a.e.}\;r\in[0,T].
		\]
		Let $(X_s)_{s\in[t_0,T]}$ be the solution to SDE~\eqref{eq:SDE_apriori01} with initial value $x_0$, and let $(X_s^n)_{s\in[t_0,T]}$ be the solution to the corresponding SDE with initial value $x_n$ and coefficients $(b^n,\sigma^n)$.
		Then, $\sup_{s\in[t_0,T]}|X_s^n-X_s|\to 0$ in probability as $n\to\infty$.
	\end{lem}
	
	\begin{proof}
		For $n,N=1,2,\dots$, set
		\[
		\tau_N^n := \inf\{s\in[t_0,T]\mid |X_s|\geq N \text{ or } |X_s^n|\geq N\}\wedge T.
		\]
		By Lemma~\ref{th:SDE_apriori_diff_loc}, there exists a constant $C_N>0$, independent of $n$, such that
		\begin{align*}
			&\E\bigg[\sup_{s\in[t_0,T]}|X_{s\wedge\tau_N^n}^n-X_{s\wedge\tau_N^n}|^2\bigg]\\
			&\leq C_N\E\bigg[|x_n-x_0|^2
			+\bigg(\int_{t_0}^{\tau_N^n}|b^n(r,X_r)-b(r,X_r)|\,dr\bigg)^2
			+\int_{t_0}^{\tau_N^n}|\sigma^n(r,X_r)-\sigma(r,X_r)|^2\,dr\bigg].
		\end{align*}
		For fixed $N$, the right-hand side tends to $0$ as $n\to\infty$ by the assumptions and the dominated convergence theorem.	
		On the other hand, by Lemma~\ref{th:SDE_apriori1}, there exists a constant $C'>0$, independent of $n$ and $N$, such that
		\begin{align*}
			\Prob(\tau_N^n<T)
			&\leq \Prob\bigg(\sup_{s\in[t_0,T]}|X_s|\geq N\bigg)
			+ \Prob\bigg(\sup_{s\in[t_0,T]}|X_s^n|\geq N\bigg)\\
			&\leq \frac{1}{N^2}\E\bigg[\sup_{s\in[t_0,T]}|X_s|^2\bigg]
			+ \frac{1}{N^2}\E\bigg[\sup_{s\in[t_0,T]}|X_s^n|^2\bigg]
			\leq \frac{C'}{N^2}(1+|x_0|^2+|x_n|^2),
		\end{align*}
		and hence $ \lim_{N\to\infty}\sup_n\Prob(\tau_N^n<T)=0 $. Now, for any $\epsilon>0$, we have
		\begin{align*}
			\Prob\bigg(\sup_{s\in[t_0,T]}|X_s^n-X_s|>\epsilon\bigg)
			&\leq \Prob\bigg(\sup_{s\in[t_0,T]}|X_{s\wedge\tau_N^n}^n-X_{s\wedge\tau_N^n}|>\epsilon\bigg)
			+ \Prob(\tau_N^n<T).
		\end{align*}
		Taking the limit superior as $n\to\infty$ and then letting $N\to\infty$ yields the desired conclusion.
	\end{proof}
	
	\section{Uniform convergence in probability of random fields}	
	The following lemma provides a criterion for the uniform convergence in probability
	of a sequence of random fields. 
	It may be compared to \cite[Theorem~4.4]{Kunita2004}, which assumes convergence in a H\"older-type norm.
	\begin{lem}\label{th:random_field_and_convergence_in_probability}
		Let $ (X(x))_{x\in\R^d} $ and $ (X_n(x))_{x\in\R^d},\;n=1,2,\dots $, be $\R^{d_0}$-valued random fields that are continuous almost surely.  
		Assume that for every $ x\in\R^d $, the sequence $ X_n(x) $ converges in probability to $ X(x) $ as $ n\to\infty $.  
		Furthermore, suppose there exist constants $ \gamma>0 $, $ \alpha>d $, and a locally bounded function $ \kappa:\R^{2d}\to[0,\infty) $ such that, for every $ n $,
		\[
		\E[|X_n(x)-X_n(y)|^\gamma]
		\leq \kappa(x,y)|x-y|^\alpha,\quad x,y\in\R^d.
		\]
		Then for every bounded set $ \mathbb{D}\subset\R^d $, we have $\sup_{x\in\mathbb{D}}|X_n(x)-X(x)| \to 0$ in probability as $n\to \infty$.
		In particular, for any $\R^d$-valued random variable $\xi$, it holds that $ X_n(\xi)\to X(\xi) $ in probability as $ n\to\infty $.
	\end{lem}
	
	\begin{proof}
		Set $ Y_n(x):=X_n(x)-X(x) $.  
		By Fatou's lemma,
		\begin{align*}
			\E[|Y_n(x)-Y_n(y)|^\gamma]
			&\leq 2^\gamma\Big(\E[|X_n(x)-X_n(y)|^\gamma]
			+\varliminf_{k\to\infty}\E[|X_k(x)-X_k(y)|^\gamma]\Big)\\
			&\leq 2^{\gamma+1}\kappa(x,y)|x-y|^\alpha.
		\end{align*}
		Choose $ \beta>0 $ such that $ 0<\beta<(\alpha-d)/\gamma $, and define
		\[
		L_n:=\sup_{x,y\in\mathbb{D},\,x\neq y}\frac{|Y_n(x)-Y_n(y)|}{|x-y|^\beta}.
		\]
		By Kolmogorov's continuity theorem, there exists a constant $ C>0 $, independent of $ n $, such that $ \E[(L_n)^\gamma]\leq C $.  
		Fix any $ \delta>0 $.  
		Since $\mathbb{D}$ is a bounded subset of $\R^d$, there exist finitely many points $ x_1,\dots,x_k\in\mathbb{D} $ such that $\mathbb{D}\subset B_\delta(x_1)\cup\cdots\cup B_\delta(x_k)$,	where $ B_\delta(x):=\{y\in\R^d\mid |y-x|<\delta\} $.  
		For each $ x\in\mathbb{D} $, there exists an index $ j_0 $ with $ x\in B_\delta(x_{j_0}) $, so that
		\begin{equation*}
			|Y_n(x)|
			\leq |Y_n(x)-Y_n(x_{j_0})|+|Y_n(x_{j_0})|
			\leq L_n\delta^\beta+\sum_{j=1}^k|Y_n(x_j)|.
		\end{equation*}
		Hence, for any $ \epsilon>0 $,
		\begin{align*}
			\Prob\bigg(\sup_{x\in\mathbb{D}}|Y_n(x)|>\epsilon\bigg)
			&\leq \Prob\Big(L_n\delta^\beta>\frac{\epsilon}{2}\Big)
			+\sum_{j=1}^k\Prob\Big(|Y_n(x_j)|>\frac{\epsilon}{2k}\Big).
		\end{align*}
		Since $\Prob(L_n\delta^\beta>\epsilon/2)\leq (2/\epsilon)^\gamma\delta^{\beta\gamma}\E[(L_n)^\gamma]\leq (2/\epsilon)^\gamma C\delta^{\beta\gamma}$,
		and since $x_1,\dots,x_k$ are independent of $n$ and $Y_n(x_j)\to 0$ in probability for each $j$, letting $n\to\infty$ gives
		\[
		\varlimsup_{n\to\infty}\Prob\bigg(\sup_{x\in\mathbb{D}}|Y_n(x)|>\epsilon\bigg)
		\leq \Big(\frac{2}{\epsilon}\Big)^\gamma C\delta^{\beta\gamma}.
		\]
		Thus, letting $\delta\downarrow0$ yields $\lim_{n\to \infty}\Prob(\sup_{x\in \mathbb{D}}|Y_n(x)|>\epsilon)= 0$,
		that is, $ \sup_{x\in\mathbb{D}}|X_n(x)-X(x)|\to0 $ in probability as $ n\to\infty $.
	\end{proof}

	\section{Weak derivatives}
	\subsection{Joint measurability of weak derivatives}
	First, using the idea of \cite[Chapter~IV, Lemma~62]{Protter2005} or \cite[Chapter~IV, Exercise~5.17]{RevuzYor1999}, 
	we state the following lemma showing that (local) convergence in measure for each fixed parameter yields a jointly measurable limit.	
	
	\begin{lem}\label{th:measurability_convergence_in_measure}
		Let $(\mathbb{X},\mathscr B,\mu)$ be a measure space, $(A,\mathscr A)$ a measurable space, and $\mathbb{B}$ a separable Banach space.  
		Let $ f_n:\mathbb{X}\times A\to\mathbb{B}\;(n=1,2,\dots)$ be $\mathscr B\otimes\mathscr A/\mathscr B(\mathbb{B})$-measurable functions.  
		Let $(\mathbb{X}_k)_{k\ge1}$ be an increasing sequence of $\mathscr B$-measurable sets with $\mathbb{X}_k\uparrow\mathbb{X}$.  
		Assume that for each $a\in A$, there exists a $\mathscr B/\mathscr B(\mathbb{B})$-measurable function $ f(\cdot,a):\mathbb{X}\to\mathbb{B} $ such that $(f_n(\cdot,a))_{n\ge1}$ converges to $f(\cdot,a)$ in measure on each $\mathbb{X}_k$.  
		Then there exists a $\mathscr B\otimes\mathscr A/\mathscr B(\mathbb{B})$-measurable function $\widetilde{f}$ such that, for every $a\in A$, we have
		\[
		\widetilde{f}(x,a)=f(x,a),\quad \mu\text{-a.e.}\;x\in\mathbb{X}.
		\]
	\end{lem}
	
	\begin{proof}
		We first consider the case $\mathbb{X}_k=\mathbb{X}$ for all $k$.  
		For each $a\in A$, define a sequence of integers $(n_k(a))_{k\ge0}$ recursively by setting $n_0(a):=1$ and
		\[
		n_k(a):=\inf\Big\{m>n_{k-1}(a)\;\Bigm|\;
		\sup_{i,j\ge m}\mu\big(\{\|f_i(\cdot,a)-f_j(\cdot,a)\|_{\mathbb{B}}>2^{-k}\}\big)\le2^{-k}\Big\},
		\quad k\geq 1.
		\]
		Then each map $a\mapsto n_k(a)$ is $\mathscr A$-measurable,  
		and by the Borel–Cantelli lemma, it follows that for every $a\in A$,  
		the sequence $(f_{n_k(a)}(x,a))_{k\ge1}$ is a Cauchy sequence in $\mathbb{B}$ for $\mu$-a.e.\;$x$.
		Define $\widetilde{f}:\mathbb{X}\times A\to\mathbb{B}$ by
		\[
		\widetilde{f}(x,a):=
		\begin{cases}
			\displaystyle\lim_{k\to\infty}f_{n_k(a)}(x,a), & \text{if the limit exists in }\mathbb{B},\\[1ex]
			0, & \text{otherwise.}
		\end{cases}
		\]
		Then $\widetilde{f}$ is $\mathscr B\otimes\mathscr A/\mathscr B(\mathbb{B})$-measurable and satisfies the desired property.
		
		In the general case, apply the above argument to the sequence $(f_n\mathbf{1}_{\mathbb{X}_k\times A})_{n\ge1}$ for each $k$.  
		This yields $\mathscr B\otimes\mathscr A/\mathscr B(\mathbb{B})$-measurable functions $\widetilde{f}^k$ such that
		\[
		\widetilde{f}^k(x,a)=f(x,a)\mathbf{1}_{\mathbb{X}_k}(x),\quad \mu\text{-a.e.}\;x,
		\]
		for every $a\in A$.  
		Finally, set
		\[
		\widetilde{f}(x,a):=
		\begin{cases}
			\displaystyle\lim_{k\to\infty}\widetilde{f}^k(x,a), & \text{if the limit exists in }\mathbb{B},\\[1ex]
			0, & \text{otherwise.}
		\end{cases}
		\]
		Then $\widetilde{f}$ satisfies the required properties.
	\end{proof}
	
	Next, we prepare a simple lemma on weak differentiability.
	
	\begin{lem}\label{th:weak_differentiability_function2}
		Let $ u:\R^d\to\R^d $ be a measurable function and $(\rho_n)_{n\ge1}$ a sequence of mollifiers on $\R^d$.  
		Then $u$ is weakly differentiable if and only if, for every $ N=1,2,\dots $, one has $ u\in L^1(B_N) $ and the sequence $(\nabla(u*\rho_n))_{n\ge1}$ converges in $L^1(B_N)$ (equivalently, is Cauchy in $L^1(B_N)$),  
		where $B_N:=\{x\in\R^d\mid |x|<N\}$.
	\end{lem}
	
	\begin{proof}
		The sufficiency follows immediately from the definition of weak differentiability and properties of mollifiers.  
		We show necessity.  
		Under the assumption, $u*\rho_n\in C^\infty(\R^d)$ for all $n$.  
		Fix $N=1,2,\dots$, and let $\varphi\in C_c^\infty(B_N)$ and $i=1,\dots,d$.  
		Then
		\begin{align*}
			\int_{B_N}(u*\rho_n)(x)\partial_i\varphi(x)\,dx
			= -\int_{B_N}(\partial_i(u*\rho_n))(x)\varphi(x)\,dx.
		\end{align*}
		Since $u*\rho_n\to u$ in $L^1(B_N)$ by the properties of mollifiers, and $(\partial_i(u*\rho_n))_{n\ge1}$ converges in $L^1(B_N)$ by assumption, we may pass to the limit as $n\to\infty$ in the above equality.  
		Thus $u$ is weakly differentiable on each $B_N$, and hence on $\R^d$.
	\end{proof}
	
	We now prove the following technical lemma concerning the joint measurability of weak derivatives.
	
	\begin{lem}\label{th:jointly_mble_vesion_of_weak_derivative}
		Let $(\Omega',\mathscr F')$ be a measurable space, and let $u:\Omega'\times\R^d\to\R^d$ be an $\mathscr F'\otimes\mathscr B(\R^d)$-measurable function.  
		Then the set
		\[
		\Gamma := \{\omega\in\Omega'\mid x\mapsto u(\omega,x)\text{ is weakly differentiable}\}
		\]
		is $\mathscr F'$-measurable.  
		Moreover, there exists an $\mathscr F'\otimes\mathscr B(\R^d)$-measurable function $v:\Omega'\times\R^d\to\R^{d\times d}$ such that
		\[
		v(\omega,x)=\nabla u(\omega,x),\quad\text{a.e. }x\in\R^d,\;\;\text{for all }\omega\in\Gamma.
		\]
	\end{lem}
	
	\begin{proof}
		Let $B_N:=\{x\in\R^d\mid |x|<N\}$ for $N=1,2,\dots$, and define
		\[
		\Gamma_0:=\Big\{\omega\in\Omega'\,\Bigm|\,
		\int_{B_N}|u(\omega,x)|\,dx<\infty \;\text{ for all } N=1,2,\dots\Big\}.
		\]
		Then $\Gamma_0\in\mathscr F'$, and for each $\omega\in\Omega'$, the map $x\mapsto u(\omega,x)\mathbf{1}_{\Gamma_0}(\omega)$ is locally integrable.  
		Let $(\rho_n)_{n\ge1}$ be a sequence of mollifiers on $\R^d$, and set
		\[
		\overline{u}_n(\omega,x):=\int_{\R^d}u(\omega,x-y)\mathbf{1}_{\Gamma_0}(\omega)\rho_n(y)\,dy,\quad \omega\in\Omega',\;x\in\R^d.
		\]
		Then $\overline{u}_n$ is $\mathscr F'\otimes\mathscr B(\R^d)$-measurable, and so is $\nabla\overline{u}_n$, when interpreted as the classical derivative. By Lemma~\ref{th:weak_differentiability_function2}, we have
		\begin{align*}
			\Gamma
			=\Gamma_0\cap
			\Big\{\omega\in\Omega'\Bigm|
			\varlimsup_{k\to\infty}\sup_{n,m\ge k}\int_{B_N}
			|\nabla\overline{u}_n(\omega,x)-\nabla\overline{u}_m(\omega,x)|\,dx=0 \;\text{ for all } N=1,2,\dots 
			\Big\},
		\end{align*}
		and hence $\Gamma\in\mathscr F'$.		
		
		Next, note that for every $\omega\in\Omega'$, the function $x\mapsto \widetilde{u}(\omega,x):=u(\omega,x)\mathbf{1}_{\Gamma}(\omega)$ is weakly differentiable.  
		Define
		\[
		\widetilde{u}_n(\omega,x):=\int_{\R^d}\widetilde{u}(\omega,x-y)\rho_n(y)\,dy.
		\]
		Since $\Gamma\in\mathscr F'$, it follows as before that $\nabla\widetilde{u}_n$ is $\mathscr F'\otimes\mathscr B(\R^d)$-measurable.  
		Moreover, for every $\omega\in\Omega'$, the properties of mollifiers ensure that 
		the sequence $(\nabla\widetilde{u}_n(\omega,\cdot))_{n\ge1}$ converges to 
		$\nabla\widetilde{u}(\omega,\cdot)$ in $L^1(B_N)$ for all $N\ge1$.  
		Hence, by Lemma~\ref{th:measurability_convergence_in_measure}, there exists an $\mathscr F'\otimes\mathscr B(\R^d)$-measurable function $v:\Omega'\times\R^d\to\R^{d\times d}$ such that
		\[
		v(\omega,x)=\nabla\widetilde{u}(\omega,x),\quad \text{a.e.}\;x\in\R^d,\;\;\text{for all }\omega\in\Omega'.
		\]
		For $\omega\in\Gamma$, the right-hand side coincides with $\nabla u(\omega,x)$ for a.e.\;$x$, and the claim follows.
	\end{proof}
	
	\subsection{Chain rule for weak derivatives}
	In general, the chain rule for weak derivatives of a composition $f(u)$ with a Lipschitz function $f:\R^{d_0}\to\R$ and a weakly differentiable function $u:\R^d\to\R^{d_0}$ fails when $d_0>1$ (see \cite{LeoniMorini2007}).  
	However, in the setting below, the chain rule holds in its classical form and can be proved by mollification.
	
	\begin{lem}\label{th:chain_rule_for_weak_derivative_Lipschitz}
		Let $f:\R^d\to\R^d$ be a Lipschitz function and $u:\R^d\to\R^d$ a weakly differentiable function.  
		Assume that the preimage of every Lebesgue null set under $u$ is again a Lebesgue null set. 
		Then $f(u)$ is weakly differentiable, and if $\widetilde{\nabla f}$ is any Borel function satisfying $\widetilde{\nabla f}=\nabla f$ a.e., then
		\[
		\nabla(f(u))=\widetilde{\nabla f}(u)\nabla u,\quad\text{a.e.}
		\]
	\end{lem}
	
	\begin{proof}
		First note that $f(u)$ is locally integrable.  
		Let $(\rho_n)_{n\ge1}$ be a sequence of mollifiers on $\R^d$ and set $f_n:=f*\rho_n$.  
		Each $f_n$ is of class $C^1$ and Lipschitz, and by an argument similar to that in \cite[Proposition~9.5]{Brezis2011} (which is stated for scalar-valued $u$), we have, for every $\varphi\in C_c^\infty(\R^d;\R^d)$,
		\begin{align}\label{eq:chain_rule_for_weak_derivative_Lipschitz}
			\int_{\R^d}f_n(u(x))\divg\varphi(x)\,dx
			=-\int_{\R^d}\nabla f_n(u(x))\nabla u(x)\varphi(x)\,dx.
		\end{align}
		Since $f$ is Lipschitz, it is weakly differentiable and its weak derivative coincides a.e.~with the classical one, hence with $\widetilde{\nabla f}$.  
		Then, by standard properties of mollifiers, it follows that 
		$\|\nabla f_n-\widetilde{\nabla f}\|_{L^1(B)}\to0$ as $n\to\infty$ for every bounded open set $B\subset\R^d$.  
		Thus, after extracting a subsequence, we may assume $\nabla f_n\to\widetilde{\nabla f}$ a.e.~on $\R^d$.  
		By the assumption on $u$, we then have $\nabla f_n(u(x))\to\widetilde{\nabla f}(u(x))$ for a.e.\;$x\in\R^d$.
		Passing to the limit $n\to\infty$ in \eqref{eq:chain_rule_for_weak_derivative_Lipschitz}, 
		we use the uniform convergence $f_n\to f$ and dominated convergence 
		(via the uniform boundedness of $(\nabla f_n)_{n\ge1}$) to obtain
		\[
		\int_{\R^d}f(u(x))\divg\varphi(x)\,dx
		=-\int_{\R^d}\widetilde{\nabla f}(u(x))\nabla u(x)\varphi(x)\,dx,
		\]
		which proves the claim.
	\end{proof}

	\section{Proof of Lemma~\ref{th:estim_of_approximation_of_quadratic_varitaions_for_g(r,X)}}\label{section:proof_of_estim_of_approximation_of_quadratic_varitaions_for_g(r,X)}
	We begin with the following auxiliary lemma concerning approximations of quadratic variations.
	
	\begin{lem}\label{th:estim_of_approximation_of_quadratic_varitaions2}
		Let $0\le t<s\le T$, and let $t=r_0<r_1<\cdots<r_n=s$ be a partition of $[t,s]$.
		\begin{enumerate}
			\item 
			Let $(A_r)_{r\in[0,T]}$ be a continuous $\R^{d_0}$-valued process of finite variation.  
			Then, for any $p>0$,
			\[
			\E\bigg[\bigg(\sum_{j=0}^{n-1}|A_{r_{j+1}}-A_{r_j}|^2\bigg)^p\bigg]
			\le \E[\|A\|_s^{2p}],
			\]
			where $(\|A\|_r)_{r\in[0,T]}$ denotes the total variation process of $(A_r)_{r\in[0,T]}$.
			
			\item 
			Let $(M_r)_{r\in[0,T]}$ be a continuous $d_0$-dimensional local martingale.  
			Then, for any $p>1/2$, there exists a constant $c_p>0$, depending only on $p$ and $d_0$, such that
			\[
			\E\bigg[\bigg(\sum_{j=0}^{n-1}|M_{r_{j+1}}-M_{r_j}|^2\bigg)^p\bigg]
			\le c_p\,\E[\langle M\rangle_s^{p}],
			\]
			where $\langle M\rangle_r := \sum_{i=1}^{d_0}\langle M^i\rangle_r$, with $M_r=(M_r^1,\dots,M_r^{d_0})$.
		\end{enumerate}
	\end{lem}
	
	\begin{proof}
		(1)  
		Since 
		$
		|A_{r_{j+1}}-A_{r_j}|^2
		\le (\|A\|_{r_{j+1}}-\|A\|_{r_j})^2
		\le (\|A\|_{r_{j+1}}-\|A\|_{r_j})\|A\|_s,
		$ 
		the desired estimate follows immediately.
		
		(2)  
		It suffices to prove the claim in the case $d_0=1$ and $M_0=0$.  
		In this case, the result follows from the classical Burkholder inequality (see, e.g., \cite[Section~11.2, Theorem~1]{ChowTeicher1997}) combined with the standard BDG inequality.
	\end{proof}
	
	We now prove Lemma~\ref{th:estim_of_approximation_of_quadratic_varitaions_for_g(r,X)}.
	
	\begin{proof}[Proof of Lemma~\ref{th:estim_of_approximation_of_quadratic_varitaions_for_g(r,X)}]
		Throughout the proof, we write $\alpha\lesssim\beta$ to mean that there exists a constant $C>0$, depending only on $p,d,\|K\|_{L^1},L$, and $T$, such that $\alpha\le C\beta$.
		Set
		\[
		G_{r}^{t,x,x'}:= g(r,X_{r}^{t,x})-g(r,X_{r}^{t,x'}).
		\]
		By the Schwarz inequality and Lemma~\ref{th:estim_of_approximation_of_quadratic_varitaions2},
		\begin{align}\label{eq:estim_of_approximation_of_quadratic_varitaions_for_g(r,X)}
			\E[|V_k^{\Delta,t,x}-V_k^{\Delta,t,x'}|^p]
			&\le \E\bigg[\bigg(\sum_{j=0}^{n-1}|G_{r_{j+1}}^{t,x,x'}-G_{r_j}^{t,x,x'}|^2\bigg)^{p}\bigg]^{1/2}
			\E\bigg[\bigg(\sum_{j=0}^{n-1}|W_{r_{j+1}}^k-W_{r_j}^k|^2\bigg)^{p}\bigg]^{1/2}\notag \\
			&\lesssim \E\bigg[\bigg(\sum_{j=0}^{n-1}|G_{r_{j+1}}^{t,x,x'}-G_{r_j}^{t,x,x'}|^2\bigg)^{p}\bigg]^{1/2}.
		\end{align}
		Next, by the fundamental theorem of calculus,
		\[
		G_r^{t,x,x'}
		=\int_0^1\langle\nabla g(r,\theta X_r^{t,x}+(1-\theta)X_r^{t,x'}),\,X_r^{t,x}-X_r^{t,x'}\rangle\,d\theta,
		\]
		and hence, using \eqref{eq:estim_of_approximation_of_quadratic_varitaions_for_g(r,X)_2}, we have
		\begin{align*}
			|G_{r_{j+1}}^{t,x,x'}-G_{r_j}^{t,x,x'}|
			&\le L\big(|r_{j+1}-r_j|+|X_{r_{j+1}}^{t,x}-X_{r_j}^{t,x}|+|X_{r_{j+1}}^{t,x'}-X_{r_j}^{t,x'}|\big)|X_{r_{j+1}}^{t,x}-X_{r_{j+1}}^{t,x'}|\\
			&\quad+L|X_{r_{j+1}}^{t,x}-X_{r_{j+1}}^{t,x'}-(X_{r_j}^{t,x}-X_{r_j}^{t,x'})|\\
			&=:A_{1,j}+A_{2,j}.
		\end{align*}
		Set 
		\[
		V_{1,p}:=\E\bigg[\bigg(\sum_{j=0}^{n-1}|A_{1,j}|^2\bigg)^p\bigg],\quad
		V_{2,p}:=\E\bigg[\bigg(\sum_{j=0}^{n-1}|A_{2,j}|^2\bigg)^p\bigg].
		\]
		By Lemma~\ref{th:estim_of_approximation_of_quadratic_varitaions2} and Lemma~\ref{th:SDE_standard_estim},
		\begin{align}
			&\begin{aligned}[b]\label{eq:estim_of_approximation_of_quadratic_varitaions_for_g(r,X)_3}
				\E\bigg[\bigg(\sum_{j=0}^{n-1}|X_{r_{j+1}}^{t,x}-X_{r_j}^{t,x}|^2\bigg)^{2p}\bigg]
				&\lesssim \E\bigg[\bigg(\int_t^s|b(r,X_r^{t,x})|\,dr\bigg)^{4p}
				+\bigg(\int_t^s|\sigma(r,X_r^{t,x})|^2\,dr\bigg)^{2p}\bigg]\\
				&\lesssim \E\Big[\Big(1+\sup_{r\in[t,T]}|X_r^{t,x}|\Big)^{4p}\Big]
				\lesssim (1+|x|)^{4p},
			\end{aligned}\\
			& \E\bigg[\bigg(\sum_{j=0}^{n-1}|r_{j+1}-r_j|^2\bigg)^{2p}\bigg]\le T^{4p},\quad
			\E\Big[\sup_{r\in[t,T]}|X_r^{t,x}-X_r^{t,x'}|^{4p}\Big]
			\lesssim |x-x'|^{4p},
			\notag
		\end{align}
		and the estimate \eqref{eq:estim_of_approximation_of_quadratic_varitaions_for_g(r,X)_3} also holds with $x$ replaced by $x'$.  
		Therefore,
		\begin{align*}
			V_{1,p}
			&\lesssim
			\E\bigg[\bigg(\sum_{j=0}^{n-1}|r_{j+1}-r_j|^2\bigg)^{2p}
			+\bigg(\sum_{j=0}^{n-1}|X_{r_{j+1}}^{t,x}-X_{r_j}^{t,x}|^2\bigg)^{2p}
			+\bigg(\sum_{j=0}^{n-1}|X_{r_{j+1}}^{t,x'}-X_{r_j}^{t,x'}|^2\bigg)^{2p}\bigg]^{1/2}\\
			&\quad\times
			\E\Big[\sup_{r\in[t,T]}|X_r^{t,x}-X_r^{t,x'}|^{4p}\Big]^{1/2}
			\lesssim (1+|x|+|x'|)^{2p}|x-x'|^{2p}.
		\end{align*}
		For $V_{2,p}$, again by Lemma~\ref{th:estim_of_approximation_of_quadratic_varitaions2} and Lemma~\ref{th:SDE_standard_estim},
		\begin{align*}
			V_{2,p}
			&\lesssim
			\E\bigg[\bigg(\sum_{j=0}^{n-1}|X_{r_{j+1}}^{t,x}-X_{r_{j+1}}^{t,x'}
			-(X_{r_j}^{t,x}-X_{r_j}^{t,x'})|^2\bigg)^p\bigg]\\
			&\lesssim
			\E\bigg[\bigg(\int_t^s|b(r,X_r^{t,x})-b(r,X_r^{t,x'})|\,dr\bigg)^{2p}
			+\bigg(\int_t^s|\sigma(r,X_r^{t,x})-\sigma(r,X_r^{t,x'})|^2\,dr\bigg)^{p}\bigg]\\
			&\lesssim
			\E\Big[\sup_{r\in[t,T]}|X_r^{t,x}-X_r^{t,x'}|^{2p}\Big]
			\lesssim |x-x'|^{2p}.
		\end{align*}
		Hence,
		\begin{align*}
			\E\bigg[\bigg(\sum_{j=0}^{n-1}|G_{r_{j+1}}^{t,x,x'}-G_{r_j}^{t,x,x'}|^2\bigg)^p\bigg]
			\lesssim
			V_{1,p} + V_{2,p}
			\lesssim (1+|x|+|x'|)^{2p}|x-x'|^{2p}.
		\end{align*}
		Returning to \eqref{eq:estim_of_approximation_of_quadratic_varitaions_for_g(r,X)}, we obtain the desired estimate.
	\end{proof}

	\section{Proof of Lemma~\ref{th:differentiability_of_sol_of_SDE}}\label{section:proof_of_th:differentiability_of_sol_of_SDE}
	To prove Lemma~\ref{th:differentiability_of_sol_of_SDE}, we use the following auxiliary result, which is a slight modification of \cite[Lemma~2.4]{Zhang2016}. Its proof is omitted.
	
	\begin{lem}\label{th:criterion_of_weak_differentiability}
		Let $ (X(x))_{x\in \R^d} $ and $ (X_n(x))_{x\in \R^d},\,n=1,2,\dots $ be $ \R^d $-valued $ \mathscr F \otimes \mathscr B(\R^d) $-measurable random fields such that for each $ n $, the map $ x\mapsto X_n(x) $ is almost surely continuous and weakly differentiable.
		Suppose furthermore that for every $ x\in \R^d $, we have $ X_n(x)\to X(x) $ in probability, and that for some $ p>1 $,
		\[
		\sup_n\Big(\E[|X_n(0)|^p] + \esssup_{x\in\R^d}\E[|\widetilde{\nabla X}_n(x)|^p]\Big)<\infty,
		\]
		where $ (\widetilde{\nabla X}_n(x))_{x\in \R^d} $ denotes the $ \mathscr F\otimes \mathscr B(\R^d) $-measurable weak derivative of $ (X_n(x))_{x\in \R^d} $ (see Definition~\ref{rem:jointly_mble_version_of_weak_derivative_process}).
		Then $ x\mapsto X(x) $ is almost surely weakly differentiable, and
		\begin{align*}
			\esssup_{x\in\R^d}\E\big[|\widetilde{\nabla X}(x)|^p\big ]
			\leq \sup_n \esssup_{x\in\R^d}\E\big [|\widetilde{\nabla X}_n(x)|^p\big ].
		\end{align*}
	\end{lem}
	
	\begin{rem}
		We use only joint measurability (not continuity) of the limit random field and work with the $\mathscr P\otimes\mathscr B(\R^d)$-measurable weak derivative as defined in Definition~\ref{rem:jointly_mble_version_of_weak_derivative_process}.
		These adjustments do not affect the proof. (While \cite{Zhang2016} is stated on the classical Wiener space, the argument itself carries over without modification.)
	\end{rem}
	
	\begin{proof}[Proof of Lemma~\ref{th:differentiability_of_sol_of_SDE}]
		Let $ (\rho_n)_{n\geq 1} $ be a sequence of mollifiers on $ \R^d $, and define
		\begin{align*}
			\overline{b}_n(r,x):=\int_{\R^d}b(r,x-y)\rho_n(y)\,dy,\quad
			\overline{\sigma}_n(r,x):= \int_{\R^d}\sigma(r,x-y)\rho_n(y)\,dy
		\end{align*}
		as well as $ b_n(r,x):=\overline{b}_n(r,x)\mathbf{1}_{\{K\leq n\}}(r)$ and $ \sigma_n(r,x):=\overline{\sigma}_n(r,x)\mathbf{1}_{\{K\leq n\}}(r) $. Then
		\begin{align}
			\begin{aligned}
				&|b_n(r,x)|\leq (K(r)\wedge n)(2+|x|), \quad |b_n(r,x)-b_n(r,x')|\leq (K(r)\wedge n)|x-x'|,\\
			&|\sigma_n(r,x)|\leq (K(r)\wedge n)^{1/2}(2+|x|),\quad
			|\sigma_n(r,x)-\sigma_n(r,x')|\leq (K(r)\wedge n)^{1/2}|x-x'|
			\end{aligned}\label{eq:differentiability_of_sol_of_SDE2}
		\end{align}
		and setting $ L_n=(n+n^{1/2})\|\nabla \rho_n\|_{L^1} $, we have
		\begin{align*}
			|\nabla b_n(r,x)- \nabla b_n(r,x')|\leq L_n|x-x'|,\quad
			|\nabla \sigma_n(r,x)- \nabla \sigma_n(r,x')|\leq L_n|x-x'|.
		\end{align*}
		
		For each $ x\in \R^d $, let $ (X_s^{x,n})_{s\in[t,T]} $ denote the solution to the SDE~\eqref{eq:differentiability_of_sol_of_SDE0} with coefficients $ b_n$ and $\sigma_n $.
		Then, by applying \cite[Chapter II, Theorems 2.2 and 3.1]{Kunita1984}\footnote{In \cite{Kunita1984}, continuity in time is also assumed for the coefficients, but all estimates remain valid under our current assumptions, and the conclusions of the theorems still hold.}, we obtain a modification $ (\widetilde{X}_s^{x,n})_{s\in[t,T],\,x\in \R^d} $ of $ (X_s^{x,n})_{s\in[t,T],\,x\in \R^d} $ satisfying the following properties:
		(1) for almost every $ \omega\in \Omega $, the map $ (s,x)\mapsto \widetilde{X}_s^{x,n}(\omega) $ is continuous and for each $ s $, the map $ x\mapsto \widetilde{X}_s^{x,n}(\omega) $ is differentiable;
		(2) for each $ x\in \R^d $ and $ i=1,\dots,d $, the partial derivative $ (\partial_i \widetilde{X}_s^{x,n})_{s\in[t,T]} $ satisfies the SDE
		\[
		\partial_i \widetilde{X}_s^{x,n} = e_i + \int_t^s \nabla b_n(r,\widetilde{X}_r^{x,n})\partial_i \widetilde{X}_r^{x,n}\, dr + \int_t^s \nabla \sigma_n(r,\widetilde{X}_r^{x,n})\partial_i \widetilde{X}_r^{x,n}\,dW_r, \quad s\in [t,T],
		\]
		where $ e_i\in \R^d $ denotes the unit vector with $1$ in the $i$-th component.
		Moreover, this modification can be chosen so that property (1) holds for all $ \omega\in\Omega $.
		
		Now, from~\eqref{eq:differentiability_of_sol_of_SDE2}, we in particular have for any $ y\in \R^d $,
		\begin{align*}
			|\nabla b_n(r,x)y|\leq K(r)|y|,\quad
			|\nabla \sigma_n(r,x)y|\leq K(r)^{1/2}|y|.
		\end{align*}
		Hence, by Lemma~\ref{th:SDE_apriori1}, for each $ p\geq 2 $, there exist constants $ C_p^{(1)},C_p^{(2)}>0 $ depending only on $ p,\|K\|_{L^1} $ such that for every $ s\in[t,T] $,
		\begin{align*}
			\E [|\widetilde{X}_s^{x,n}|^p]
			\leq C_p^{(1)}(1+|x|^p),\quad
			\E[|\partial_i \widetilde{X}_s^{x,n}|^p]
			\leq C_p^{(2)}.
		\end{align*}
		Moreover, by properties of mollifiers and Lemma~\ref{th:SDE_stability}, we have $ \widetilde{X}_s^{x,n} \to X_s^x $ in probability as $ n\to\infty $ for each $ x\in\R^d $. Hence, by Lemma~\ref{th:criterion_of_weak_differentiability}, the map $ x\mapsto X_s^x $ is weakly differentiable almost surely, and
		\[
		\E\big[|\widetilde{\nabla X}{}_s^x|^p\big]
		\leq \sup_n \esssup_{x'\in\R^d}\E\big[|{\nabla \widetilde{X}}{}_s^{x',n}|^p\big]
		\leq d^{p/2} \sum_{i=1}^d\sup_n\sup_{x'\in\R^d}\E\big[|{\partial_i \widetilde{X}}{}_s^{x',n}|^p\big]
		\leq d^{p/2+1}C_p^{(2)}
		\]
		for almost every $x\in \R^d$.
		The estimate for $ \E [|X_s^x|^p ] $ follows from Lemma~\ref{th:SDE_apriori1}, and the final claim follows from Fubini's theorem.
	\end{proof}

	\section{Supplement to Step~2 in the proof of Theorem~\ref{th:main_theorem_locally_Lipschitz}}\label{section:appendix_to_proof_of_main_theorem}
	We verify here that each condition stated in Step~2 of the proof of Theorem~\ref{th:main_theorem_locally_Lipschitz} indeed holds.
	
	\begin{proof}[Proof of Step~2]
		Note that,  by the definition of $\gamma$ (see the beginning of Step~2), we have
		\[
		(1+|x|)|\nabla\chi_n(x)|\le\gamma,\quad
		(1+|x|)^2|\nabla^2\chi_n(x)|\le\gamma,\quad x\in \R^d.
		\]
		
		\ref{item:main_theorem_locally_Lipschitz1}  
		The global Lipschitz continuity of the coefficients follows from Lemma~\ref{th:approximation_of_locally_Lipschitz_func} below.  
		We only check the linear growth condition for $\widehat{b}^n$. By \ref{item:asummption_of_main_theorem_2} and \ref{item:asummption_of_main_theorem_4}, we have  
		\begin{align*}
			|\widehat{b}^n(r,x)|
			&\le |b(r,x)| + |\widehat{\sigma}(r,x)||\chi_n(x)|^2 + |\sigma(r,x)|^2|\nabla\chi_n(x)|\\
			&\le (K(r)+\widehat{K}_{2n}(r)+\gamma K(r))(1+|x|)
			\le \widetilde{K}_n(r)(1+|x|).
		\end{align*}
		
		\ref{item:main_theorem_locally_Lipschitz2}  
		The case $B=\{x\in\R^d\mid |x|\le n\}$ is almost immediate, so we prove only the case $B=\R^d$. 
		For \eqref{eq:main_theorem_locally_Lipschitz02}, it suffices to check the condition for $\widehat{b}^n$.  
		By \eqref{eq:main_theorem_locally_Lipschitz'} and the linear growth condition on $\sigma$ stated in \ref{item:asummption_of_main_theorem_2}, for any $r\in[0,T]$ and $x\in\R^d$,
		\begin{align*}
			|\langle\widehat{b}^n(r,x),\nabla\rho(x)\rangle|
			&\le |\langle b(r,x)-\widehat{\sigma}(r,x),\nabla\rho(x)\rangle|
			+|\sigma(r,x)||\nabla\chi_n(x)|\cdot|\sigma(r,x)||\nabla\rho(x)|\\
			&\le (\widetilde{K}(r)+\gamma K(r)^{1/2}\widetilde{K}(r)^{1/2})\rho(x)
			\le \widetilde{K}_n(r)\rho(x).
		\end{align*}
		
		Next, we prove \eqref{eq:main_theorem_locally_Lipschitz03}.  
		By direct computation of $\tr\nabla\widehat{b}^n$ and $\tr[(\nabla\sigma_{(k)}^n)^2]$, one finds that for a.e.\;$r\in[0,T]$ and a.e.\;$x$,
		\begin{align*}
			&\Big|-\tr\nabla\widehat{b}^n(r,x)
			-\frac12\sum_{k=1}^{d'}\tr[(\nabla\sigma_{(k)}^n(r,x))^2]\Big|\\
			&\le\Big|-\tr\nabla b(r,x)
			+\tr\nabla\widehat{\sigma}(r,x)
			-\frac12\sum_{k=1}^{d'}\tr[(\nabla\sigma_{(k)}(r,x))^2]\Big|\chi_n(x)^2
			+V^n(r,x),
		\end{align*}
		where, setting $\widecheck{\sigma}_j(r,x):=\sum_{i,k}\sigma_{jk}(r,x)\partial_i\sigma_{ik}(r,x)$ and $\widecheck{\sigma}:=(\widecheck{\sigma}_1,\dots,\widecheck{\sigma}_d)$,
		\begin{align*}
			V^n(r,x):={}&
			\Big|\langle -b(r,x)+\widehat{\sigma}(r,x),\nabla\chi_n(x)\rangle\chi_n(x)
			+\langle\widecheck{\sigma}(r,x),\nabla\chi_n(x)\rangle\chi_n(x)\\
			&\quad+\frac12|\sigma(r,x)^\top\nabla\chi_n(x)|^2
			+\tr[\sigma(r,x)\sigma(r,x)^\top\nabla^2\chi_n(x)]\chi_n(x)\Big|.
		\end{align*}
		Note that $|\widecheck{\sigma}(r,x)|\le|\sigma(r,x)|\big(\sum_{k=1}^{d'}(\tr\nabla\sigma_{(k)}(r,x))^2\big)^{1/2}\le|\sigma(r,x)|\widetilde{K}(r)^{1/2}$.  
		Then, using the linear growth assumptions \ref{item:asummption_of_main_theorem_2} and \ref{item:asummption_of_main_theorem_4}, we obtain
		\begin{align*}
			V^n(r,x)
			&\le\gamma\big(2K(r)+2\widehat{K}_{2n}(r)+K(r)^{1/2}\widetilde{K}(r)^{1/2}\big)
			+\frac12\gamma^2K(r)+\gamma K(r)\\
			&\le\Big(\frac12\gamma^2+4\gamma\Big)K(r)
			+\gamma\widetilde{K}(r)
			+2\gamma\widehat{K}_{2n}(r).
		\end{align*}
		Combining this with \eqref{eq:main_theorem_locally_Lipschitz}, we find that for a.e.\;$r\in[0,T]$ and a.e.\;$x$,
		\begin{align*}
			&\Big|-\tr\nabla\widehat{b}^n(r,x)
			-\frac12\sum_{k=1}^{d'}\tr[(\nabla\sigma_{(k)}^n(r,x))^2]\Big|\\
			&\le\widetilde{K}(r) +\Big(\frac12\gamma^2+4\gamma\Big)K(r)
			+\gamma\widetilde{K}(r)
			+2\gamma\widehat{K}_{2n}(r)
			\le\widetilde{K}_n(r).
		\end{align*}
		The estimate for $\sum_{k=1}^{d'}(\tr\nabla\sigma_{(k)}^n(r,x))^2$ can be shown in a similar way.
		
		\ref{item:main_theorem_locally_Lipschitz4}  
		The first part follows from Lemma~\ref{th:approximation_of_locally_Lipschitz_func} below, and the second part is immediate.
	\end{proof}
	
	\begin{lem}\label{th:approximation_of_locally_Lipschitz_func}
		Let $f$ be a function on $\R^d$ with values in $\R^d$ or $\R^{d\times d'}$.  
		Assume that for every $N=1,2,\dots$, there exist constants $\kappa_N,\kappa_N'\ge0$ such that
		\[
		|f(x)|\le\kappa_N'(1+|x|),\quad
		|f(x)-f(x')|\le\kappa_N|x-x'|,\qquad |x|,|x'|\le N.
		\]
		Let $\chi\in C_c^\infty(\R^d)$ satisfy $0\le\chi\le1$ and $\chi(x)=1$ for $|x|\le1$, $\chi(x)=0$ for $|x|\ge2$.  
		For $n=1,2,\dots$, set $\chi_n(x):=\chi(x/n)$ and $f_n(x):=f(x)\chi_n(x)$.  
		Then the following estimates hold:
		\begin{align*}
			&|f_n(x)-f_n(x')|
			\le(\sqrt{d}\kappa_{2n+1}+3\|\nabla\chi\|_\infty\kappa_{2n}')|x-x'|,\quad x,x'\in\R^d,\\
			&|f_n(x)-f_n(x')|
			\le(\sqrt{d}\kappa_{N+1}+3\|\nabla\chi\|_\infty\kappa_{N}')|x-x'|,\quad |x|,|x'|\le N.
		\end{align*}
		Moreover, when $f$ takes values in $\R^{d\times d'}$, we define $\widehat{f}_n(x):=f(x)f(x)^\top\nabla\chi_n(x)\chi_n(x)$. Then
		\[
		|\widehat{f}_n(x)-\widehat{f}_n(x')|
		\le\big(6\sqrt{d}\|\nabla\chi\|_\infty\kappa_{2n+1}\kappa_{2n}'
		+9(\|\nabla^2\chi\|_\infty+\|\nabla\chi\|_\infty^2)(\kappa_{2n}')^2\big)|x-x'|,
		\quad x,x'\in\R^d.
		\]
	\end{lem}
	
	\begin{proof}
		If $f$ is of class $C^1$, estimates of the same (in fact, sharper) form follow by directly bounding $|\nabla f_n|$ and $|\nabla\widehat{f}_n|$.  
		In the general case, approximate $f$ by the mollified functions $f^{(j)}:=f*\rho_j$  (where $(\rho_j)_{j\ge1}$ is a sequence of mollifiers), apply the previous argument to each $f^{(j)}$, and then pass to the limit as $j\to\infty$ to obtain the stated bounds.
	\end{proof}

\medskip
\noindent
\textbf{Acknowledgements.}
The author thanks Masaaki Fukasawa for valuable advice on the submission of this manuscript.

\end{document}